\documentclass[12pt,a4paper]{article}

\usepackage{tikz}
\usepackage{amsmath,amsfonts,amsthm,amssymb}
\usepackage{stmaryrd} % for graph Cartesian product symbol 

\usepackage{geometry}
\usepackage{hyperref}

\newtheorem{theorem}{Theorem}[section]
\newtheorem{lemma}[theorem]{Lemma}
\newtheorem{observation}[theorem]{Observation}
\newtheorem{proposition}[theorem]{Proposition}
\newtheorem{corollary}[theorem]{Corollary}
\newtheorem{conjecture}[theorem]{Conjecture}
\newtheorem{question}[theorem]{Question}
\theoremstyle{definition}
\newtheorem{definition}[theorem]{Definition}
\newtheorem{example}[theorem]{Example}
\newtheorem{remark}[theorem]{Remark}
\usetikzlibrary{calc}

\DeclareMathOperator{\sig}{sig}
\DeclareMathOperator{\Sig}{Sig}
\DeclareMathOperator{\unsat}{unsat}
\DeclareMathOperator{\supp}{supp}

\newcommand{\tree}[1]{\mathrm{Tree}(#1)}
\newcommand{\dir}[1]{\mathrm{dir}(#1)}
\newcommand{\dd}[1]{\mathrm{dd}(#1)}

\newcommand{\eslide}[1]{\ensuremath{\mathcal{E}(Q_{#1})}}
\newcommand{\eslidesig}[2][]{\ensuremath{\mathcal{E}_{#1}(#2)}}

\newcommand{\stacking}[2]{\genfrac{}{}{0pt}{}{#1}{#2}}
\newcommand{\power}[1]{\mathcal{P}(#1)}
\newcommand{\powernon}[1]{\mathcal{P}_{\geq 1}^{#1}}
\newcommand{\ptwo}[1]{\mathcal{P}_{\geq 2}^{#1}}

\newcommand{\naturalnumbers}{\ensuremath{\mathbb{N}}}
\newcommand{\eps}{\varepsilon}
\newcommand{\excess}[2]{\ensuremath{\eps_{#1}^{#2}}}
\newcommand{\supersat}[1][n]{\ensuremath{\mathcal{SS}_{#1}}}
\newcommand{\capr}[1]{\ensuremath{\pi_{#1}}}
\newcommand{\contract}[2][n]{\ensuremath{Q_{#1}/{#2}}}
\newcommand{\contractb}[2][n]{\ensuremath{Q_{#1}/\bar{#2}}}
\newcommand{\red}[2][n]{\ensuremath{\mathrm{RSig}_{#2}(Q_{#1})}}
\newcommand{\rtree}[2][n]{\ensuremath{\mathrm{RTree}_{#2}(Q_{#1})}}

\newlength{\hspacing}
\newlength{\vspacing}

\begin{document}

\title{Characterisation and classification of signatures of spanning trees of the $n$--cube}

\author{Howida A.\ Al Fran, David J.\ W.\ Simpson and Christopher P.\ Tuffley
\\ Massey University, New Zealand}

\maketitle

\begin{abstract} 
The signature of a spanning tree $T$ of the $n$-cube $Q_n$ is the $n$--tuple 
\[
\sig(T)=(a_1,a_2,\dots,a_n)
\]
such that $a_i$ is the number of edges of $T$ in the $i$th direction. We characterise the $n$--tuples that can occur as the signature of a spanning tree, and classify a signature $\mathcal{S}$ as \emph{reducible} or \emph{irreducible} according to whether or not there is a proper nonempty subset $R$ of $[n]$
such that restricting $\mathcal{S}$ to the indices in $R$ gives a signature of $Q_{|R|}$. 
If so, we say moreover that $\mathcal{S}$ and $T$ \emph{reduce over} $R$. 

We show that reducibility places strict structural constraints on $T$. In particular, if $T$ reduces over a set of size $r$ then $T$ decomposes as a sum of $2^r$ spanning trees of $Q_{n-r}$, together with a spanning tree of a certain contraction of $Q_n$ with underlying simple graph $Q_r$. 
Moreover, this decomposition is realised by an isomorphism of \emph{edge slide graphs}, where the edge slide graph of $Q_n$ is the graph $\eslide{n}$ 
on the spanning trees of $Q_n$, 
with an edge between two trees if and only if they are related by an \emph{edge slide}. An edge slide is an operation on spanning trees of the $n$--cube given by ``sliding'' an edge of a spanning tree across a $2$--dimensional face of the cube to get a second spanning tree. 

The signature of a spanning tree is invariant under edge slides, so the subgraph $\eslidesig{\mathcal{S}}$ of $\eslide{n}$ induced by the trees with signature $\mathcal{S}$ is a union of one or more connected components of $\eslide{n}$. Reducible signatures may be further divided into \emph{strictly reducible} and \emph{quasi-irreducible} signatures, and 
as an application of our results we show that $\eslidesig{\mathcal{S}}$ is disconnected if $\mathcal{S}$ is strictly reducible. We conjecture that the converse is also true. 
If true, this would imply that the connected components of \eslide{n}\ can be characterised in terms of signatures of spanning trees of subcubes.
\end{abstract}

\section{Introduction}
The $n$--cube is the graph $Q_n$ whose vertices are the subsets of the set $[n]=\{1,2,\dots,n\}$, with an edge between $X$ and $Y$ if they differ by the addition or removal of a single element. The element added or removed is the \emph{direction} of the edge. Given a spanning tree $T$ of $Q_n$, we may then define the \emph{signature} of $T$ to be the $n$--tuple
\[
\sig(T) = (a_1,a_2,\dots,a_n),
\]
where $a_i$ is the number of edges of $T$ in direction $i$. The signature of $T$ carries exactly the same information as the \emph{direction monomial} $q^{\dir{T}}$ appearing in Martin and Reiner's weighted count~\cite{martin-reiner-2003} of the spanning trees of $Q_n$. With respect to certain weights $q_1,\dots,q_n$ and $x_1,\dots,x_n$ they show that
\[
\sum_{T\in\tree{Q_n}} q^{\dir{T}}x^{\dd{T}}
    =q_1\cdots q_n \prod_{\stacking{S\subseteq [n]}{|S|\geq 2}}
                \sum_{i\in S} q_i (x_i^{-1}+x_i),
\]
where 
\[
q^{\dir{T}} = q_1^{a_1}q_2^{a_2}\dots q_n^{a_n}.
\]
Thus, the signature and direction monomial completely determine each other.
(The second factor $x^{\dd{T}}$ appearing here is the \emph{decoupled degree monomial} of $T$. It plays no role in this paper, so we refer the interested reader to Martin and Reiner~\cite{martin-reiner-2003} for the definition, and Tuffley~\cite[Sec.\ 2.2]{tuffley-2012} for an alternate formulation in terms of a canonical orientation of the edges of $T$.)

The goal of this paper is to study the signatures of spanning trees of $Q_n$, and to understand what $\sig(T)$ tells us about the structure of $T$. 
We begin by using Hall's Theorem to characterise the $n$--tuples that can occur as the signature of a spanning tree of $Q_n$. We then classify $T$ and $\mathcal{S}=\sig(T)$ as \emph{reducible} or \emph{irreducible} according to whether or not there is a proper nonempty subset $R$ of $[n]$ such that restricting $R$ to the indices in $S$ gives a signature of $Q_{|R|}$. We say that such a set $R$ is a \emph{reducing set} for $\mathcal{S}$, and that $T$ and $\mathcal{S}$ \emph{reduce over $R$}. Each signature $\mathcal{S}$ has an \emph{unsaturated part} $\unsat(\mathcal{S})$, and we further classify reducible signatures as \emph{strictly reducible} or \emph{quasi-irreducible} according to whether or not $\unsat(\mathcal{S})$ is reducible or irreducible. 

We show that reducibility places strict structural constraints on $T$. In particular, if $T$ reduces over $R$ then $T$ decomposes as a sum of a spanning tree $T^X$ of $Q_{[n]-R}$ for each $X\subseteq R$, together with a spanning tree $T_R$ of the multigraph $\contractb{R}$ obtained by contracting every edge of $Q_n$ in directions belonging to $\bar{R}=[n]-R$. 
The graph $\contractb{R}$ has underlying simple graph $Q_{|R|}$, and $2^{n-|R|}$ parallel edges for each edge of $Q_{|R|}$. 
Moreover, this decomposition may be realised as an isomorphism of \emph{edge slide graphs}. 

An \emph{edge slide} is an operation on spanning trees of $Q_n$, in which an edge of a spanning tree $T$ is ``slid'' across a $2$--dimensional face of $Q_n$ to get a second spanning tree $T'$. The \emph{edge slide graph} of $Q_n$ is the graph $\eslide{n}$ with vertices the spanning trees of $Q_n$, and an edge between two trees if they are related by an edge slide. 
Edge slides were introduced by the third author~\cite{tuffley-2012} as a means to combinatorially count the spanning trees of $Q_3$, and thereby answer the first nontrivial case of a
question implicitly raised by Stanley.
The number of spanning trees of $Q_n$ is known by Kirchhoff's Matrix Tree Theorem to be 
\[
|\text{Tree}(Q_n)|=2^{2^n-n-1}\prod_{k=1}^{n}k^{n \choose k}
\]
(see for example Stanley~\cite{stanley-1999}), and Stanley implicitly asked for a combinatorial proof of this fact. Tuffley's method to count the spanning trees of $Q_3$ using edge slides does not readily extend to higher dimensions, but the edge slide graph may nevertheless carry insight into the structure of the spanning trees of $Q_n$. 
Stanley's question has since been answered in full by Bernardi~\cite{bernardi-2012}.

In particular, a natural question of interest is to determine the connected components of $\eslide{n}$.  The signature is easily seen to be constant on connected components, and consequently the subgraph \eslidesig{\mathcal{S}}\ induced by the spanning trees with signature $\mathcal{S}$ is a union of connected components of \eslide{n}. We say that a signature $\mathcal{S}$ is \emph{connected} if $\eslidesig{\mathcal{S}}$ is connected, and \emph{disconnected} otherwise. We conclude the paper by using our results to show that all strictly reducible signatures are disconnected, and conjecture that $\mathcal{S}$ is connected if and only if $\mathcal{S}$ is irreducible or quasi-irreducible. 
If true, this would imply that the connected components of \eslide{n}\ can be characterised in terms of signatures of spanning trees of subcubes.
We show that it suffices to consider the irreducible case only.

\subsection{Organisation} 

The paper is organised as follows. 
Section~\ref{sec:defn} sets out the bulk of the definitions and notation needed for the paper,
with the introduction of some further definitions not needed until Section~\ref{sec:structural} postponed until then. We characterise signatures of spanning trees of $Q_n$ in Section~\ref{sec:characterisation}, and classify them in Section~\ref{sec:classification}. In Sections~\ref{sec:uprighttrees} and~\ref{sec:structural} we study the structural consequences of reducibility, considering first upright trees in Section~\ref{sec:uprighttrees} and then arbitrary reducible trees in Section~\ref{sec:structural}. We then use our results from Section~\ref{sec:structural} to prove that strictly reducible signatures are disconnected in 
Section~\ref{sec:disconnected}, and conclude with a discussion in Section~\ref{sec:discussion}. 

\section{Definitions and notation I}
\label{sec:defn}

This section sets out some definitions and notation used throughout the paper. 
Some additional definitions not needed until Section~\ref{sec:structural} are set out in a second definitions section there.

\subsection{General notation}

Given a graph $G$ we denote the vertex set of $G$ by $V(G)$ and the edge set of $G$ by $E(G)$. We write $\tree{G}$ for the set of spanning trees of $G$.

Given a set $S$, we denote the power set of $S$ by $\power{S}$. For $1\leq k\leq |S|$ we write
\[
\mathcal{P}_{\geq k}(S) = \{X\subseteq S: |X|\geq k\}. 
\]
For $n\in\naturalnumbers$ we let $[n]=\{1,2,\ldots,n\}$, 
and also write $\mathcal{P}_{\geq k}^{n}$ for $\mathcal{P}_{\geq k}([n])$.
For example, $\mathcal{P}_{\geq 2}^{3}=\{\{1,2\},\{1,3\},\{2,3\},\{1,2,3\}\}$.

\subsection{The $n$--cube}

\begin{definition}
We regard the \textbf{$n$--dimensional cube} or \textbf{$n$--cube} as the graph $Q_n$ with vertex set the power set of $[n]$, and an edge between vertices $X$ and $Y$ if and only if they differ by adding or removing exactly one element. The \textbf{direction} of the edge $e=\{X,Y\}$ is the unique element $i$ such that $X\oplus Y=\{i\}$, where $\oplus$ denotes symmetric difference.

For any $S\subseteq [n]$, we define $Q_S$ to be the induced subgraph of $Q_n$ with vertices the subsets of $S$. Observe that $Q_S$ is an $|S|$--cube.
\end{definition}

\subsection{The signature of a spanning tree of $Q_n$}

\begin{definition} 
Given a spanning tree $T$ of $Q_n$, the \textbf{signature} of $T$ is the
$n$--tuple 
\[
\sig(T)=(a_1, a_2, \dots, a_n),
\]
where for each $i$ the entry $a_i$ is the number of edges of $T$ in direction $i$. We will say that $\mathcal{S}=(a_1,\dots,a_n)$ is a \emph{signature of $Q_n$} if there is a spanning tree $T$ of $Q_n$ such that $\sig(T)=\mathcal{S}$, and we let
\[
\Sig(Q_n)= \{\sig(T): T\in\tree{Q_n}\}.
\]
\end{definition}

Figure~\ref{fig:trees} shows a pair of spanning trees of $Q_3$ with signature $(2,2,3)$. We note that the signature of $T$ carries exactly the same information as the \emph{direction monomial} $q^{\dir{T}}$ of Martin and Reiner~\cite{martin-reiner-2003}, because
\[
q^{\dir{T}} = q_1^{a_1}q_2^{a_2}\dots q_n^{a_n}
\qquad\Leftrightarrow\qquad
\sig(T) = (a_1,a_2,\dots,a_n). 
\]
The entries of $\sig(T)$ satisfy $1\leq a_i\leq 2^{n-1}$, because $Q_n$ has $2^{n-1}$ edges in direction $i$ and deleting them disconnects $Q_n$, and
\[
\sum_{i=1}^n a_i = |E(T)|=2^n-1.
\]
These conditions are not sufficient conditions for an $n$--tuple $(a_1, a_2, \dots a_n)$ to be a signature of $Q_n$. We find necessary and sufficient conditions in Section~\ref{sec:characterisation}. 

If $\mathcal{S}=(a_1,\dots,a_n)$ is a signature of $Q_n$ then so is any permutation of $\mathcal{S}$, because any permutation of $[n]$ induces an automorphism of $Q_n$. It follows that $\mathcal{S}$ is a signature if and only if the $n$--tuple $\mathcal{S}'$ obtained by permuting $\mathcal{S}$ to nondecreasing order is a signature. Accordingly we make the following definition:
\begin{definition}
A signature $(a_1, a_2, \dots, a_n)$ of $Q_n$ is \textbf{ordered} if $a_1\leq a_2\leq \dots\leq a_n$.
\end{definition} 
We will characterise signatures by characterising ordered signatures. 

\begin{figure}
\begin{center}
\setlength{\vspacing}{1.6cm}
\setlength{\hspacing}{1.8cm}
\begin{tabular}{ccc}
\begin{tikzpicture}[vertex/.style={circle,draw,inner sep=1pt,minimum size=2.5mm},thick]
\node (123) at (0,3*\vspacing) [vertex] {};
\node (12)  at (-\hspacing,2*\vspacing) [vertex]  {};
\node (13)  at (0,2*\vspacing) [vertex]  {};
\node (23)  at (\hspacing,2*\vspacing) [vertex]  {};
\node (1)  at (-\hspacing,\vspacing) [vertex]  {};
\node (2)  at (0,\vspacing) [vertex]  {};
\node (3)  at (\hspacing,\vspacing) [vertex]  {};
\node (0)  at (0,0) [vertex] {};
\foreach \x/\Y in {0/{1,2,3},1/{12,13},2/{12,23},3/{13,23},123/{12,13,23}}
   {\foreach \y in \Y
       \draw (\x) -- (\y);};
\draw [color=white,line width=4] (2)--(12);
\draw  (2)--(12);
\draw [color=white,line width=6pt] (2)--(23);
\draw  (2)--(23);
\node [right] at (0) {$\;\emptyset$};
\node [right] at (3) {$\{3\}$};
\node [right] at (23) {$\{2,3\}$};
\node [right] at (123) {$\{1,2,3\}$};
\node [left] at (12) {$\{1,2\}$};
\node [below right] at (2) {$\{2\}$};
\node [left] at (1) {$\{1\}$};
\node [right] at (13) {$\{1,3\}$};
\draw [line width=3] (12) -- (123);
\draw [line width=3] (1) -- (13);
\draw [line width=3] (0) -- (3);
\draw [line width=3] (3) -- (23);
\draw [line width=3,color=blue] (1) -- (12);
\draw [color=white,line width=5] (2)--(12);
\draw [line width=3] (12) -- (2);
\draw [line width=3] (1) -- (0);
\end{tikzpicture}& \hspace{\hspacing} &
\begin{tikzpicture}[vertex/.style={circle,draw,inner sep=1pt,minimum size=2.5mm},thick]
\node (123) at (0,3*\vspacing) [vertex] {};
\node (12)  at (-\hspacing,2*\vspacing) [vertex]  {};
\node (13)  at (0,2*\vspacing) [vertex]  {};
\node (23)  at (\hspacing,2*\vspacing) [vertex]  {};
\node (1)  at (-\hspacing,\vspacing) [vertex]  {};
\node (2)  at (0,\vspacing) [vertex]  {};
\node (3)  at (\hspacing,\vspacing) [vertex]  {};
\node (0)  at (0,0) [vertex] {};
\foreach \x/\Y in {0/{1,2,3},1/{12,13},2/{12,23},3/{13,23},123/{12,13,23}}
   {\foreach \y in \Y
       \draw (\x) -- (\y);};
\draw  (2)--(12);
\draw [color=white,line width=6pt] (2)--(23);
\draw  (2)--(23);
\node [right] at (0) {$\;\emptyset$};
\node [right] at (3) {$\{3\}$};
\node [right] at (23) {$\{2,3\}$};
\node [right] at (123) {$\{1,2,3\}$};
\node [left] at (12) {$\{1,2\}$};
\node [below right] at (2) {$\{2\}$};
\node [left] at (1) {$\{1\}$};
\node [right] at (13) {$\{1,3\}$};
\draw [line width=3] (12) -- (123);
\draw [line width=3] (1) -- (13);
\draw [line width=3] (0) -- (3);
\draw [line width=3] (3) -- (23);
\draw [line width=3,blue] (0) -- (2);
\draw [color=white,line width=5] (2)--(12);
\draw [line width=3] (12) -- (2);
\draw [line width=3] (1) -- (0);
\end{tikzpicture}
\end{tabular}
\caption{A pair of spanning trees of $Q_3$ with signature $(2,2,3)$. The two trees are related by an edge slide in direction~1 (Section~\ref{sec:edgeslides}): the tree on the right is obtained from the tree on the left by deleting the edge $\{\{1\},\{1,2\}\}$, and replacing it with the edge $\{\emptyset,\{2\}\}$.
 The tree on the right is upright (Section~\ref{sec:upright}), with associated section defined by $\psi_T([3]) = 3$, $\psi_T(\{1,2\})= 1$, $\psi_T(\{1,3\})= 3$, $\psi_T(\{2,3\})=2$, and $\psi_T(\{i\})=i$ for $1\leq i\leq 3$.}
\label{fig:trees}
\end{center}
\end{figure}
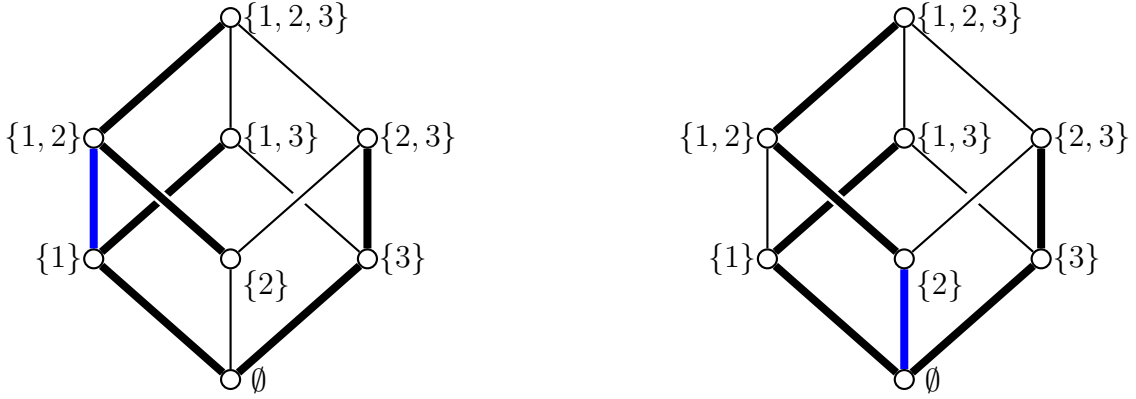

\begin{example}[Signatures in low dimensions]
\label{ex:lowdimsigs}
The 1-cube $Q_1$ has a unique spanning tree, with signature $(1)$.  The 2-cube has a total of four spanning trees: two with each of the signatures $(1,2)$ and $(2,1)$. The $3$--cube $Q_3$ has three signatures up to permutation, namely
$(1,2,4)$, $(1,3,3)$ and $(2,2,3)$. There are $16$ spanning trees with signature $(1,2,4)$; $32$ with signature $(1,3,3)$; and $64$ with signature $(2,2,3)$, for a total of $6\cdot 16+3\cdot 32+3\cdot 64=384$ spanning trees of $Q_3$. 
\end{example}

\subsection{Edge slides and the edge slide graph}
\label{sec:edgeslides}

For each $i\in[n]$ we define $\sigma_i$ to be the automorphism of $Q_n$ defined for each $X\in\power{[n]}$ by
\[
\sigma_i(X) = X\oplus\{i\},
\]
where $\oplus$ denotes symmetric difference.

\begin{definition}[Tuffley~\cite{tuffley-2012}]
Let $T$ be a spanning tree of $Q_n$, and let $e$ be an edge of $T$
in a direction $j\neq i$ such that $T$ does not also contain $\sigma_i(e)$.
We say that $e$ is \emph{$i$--slidable} or \emph{slidable in
direction $i$} if deleting $e$ from $T$ and replacing it with $\sigma_i(e)$
yields a second spanning tree $T'$; that is, if $T'=T-e+\sigma_i(e)$ is a spanning tree.
\end{definition}

\begin{example}
\label{ex:edgeslide}
Figure~\ref{fig:trees} illustrates an edge slide. 
The tree on the right is obtained from the tree on the left by 
deleting the edge $e=\{\{1\},\{1,2\}\}$, and replacing it with the edge 
$\sigma_1(e)=\{\emptyset,\{2\}\}$. This constitutes an edge slide in direction~1. We visualise it as
``sliding'' the edge $e$ in direction~1 across the $2$--dimensional face induced by
$\bigl\{\emptyset,\{1\},\{2\},\{1,2\}\bigr\}$.
\end{example}

Edge slides are a specialisation of Goddard and Swart's \emph{edge move}~\cite{goddard-1996} to the $n$--cube, in which the edges involved in the move are constrained by the structure of the cube. We visualise them as the operation of 
``sliding'' an edge across a $2$--dimensional face of the cube to get a second spanning tree, as seen in Example~\ref{ex:edgeslide} and Figure~\ref{fig:trees}.

Slidable edges may be characterised as follows:

\begin{lemma}
\label{lem:slidable}
Let $T$ be a spanning tree of $Q_n$, and let $e$ be an edge of $T$
in direction $j\neq i$. Then $e$ is $i$--slidable if and only if $\sigma_i(e)$ does not belong to $T$, and the cycle $C$ in $T+\sigma_i(e)$ created by adding $\sigma_i(e)$ to $T$ contains both $e$ and $\sigma_i(e)$, and so is broken by deleting $e$.
\end{lemma}

We define the \emph{edge slide graph} of $Q_n$ in terms of edge slides:

\begin{definition}[Tuffley~\cite{tuffley-2012}]
The \textbf{edge slide graph} of $Q_n$ is the graph $\eslide{n}$ with vertex set $\tree{Q_n}$, and an edge between trees $T_1$ and $T_2$ if and only if $T_2$ may be obtained from $T_1$ by a single edge slide. 
\end{definition}

For a connected graph $G$ the \emph{tree graph}~\cite{goddard-1996} of $G$ is the graph $T(G)$ on the spanning trees of $G$, with an edge between two trees if they're related by an edge move. 
The edge slide graph $\eslide{n}$ is therefore a subgraph of the tree graph $T(Q_n)$. 
The tree graph $T(Q_n)$ is connected, because 
$T(G)$ is easily shown to be connected for any connected graph $G$. In contrast, $\eslide{n}$ is disconnected for all $n\geq2$: 
edge slides do not change the signature, so the signature is constant on connected components. Accordingly, we make the following definition:

\begin{definition}
Let $\mathcal{S}$ be a signature of $Q_n$. The \textbf{edge slide graph} of $\mathcal{S}$ is the subgraph $\eslidesig{\mathcal{S}}$ of $\eslide{n}$ induced by the spanning trees with signature~$\mathcal{S}$. If $\mathcal{X}$ is a set of signatures, we further define
\[
\eslidesig{\mathcal{X}}=\bigcup_{\mathcal{S}\in\mathcal{X}}\eslidesig{\mathcal{S}}.
\]
\end{definition}

By our discussion above, for each signature $\mathcal{S}$ the edge slide graph $\eslidesig{\mathcal{S}}$ is a union of one or more connected components of $\eslide{n}$. We say that $\mathcal{S}$ is \textbf{connected} or \textbf{disconnected} according to whether $\eslidesig{\mathcal{S}}$ is connected or disconnected. In Section~\ref{sec:classification} we classify signatures as irreducible, quasi-irreducible or strictly reducible. We prove in Theorem~\ref{thm:disconnected} that every strictly reducible signature is disconnected, and conjecture that $\mathcal{S}$ is connected if and only if $\mathcal{S}$ is irreducible or quasi-irreducible. 
If true, this would imply that the connected components of \eslide{n}\ can be characterised in terms of signatures of spanning trees of subcubes.
By Theorem~\ref{thm:connectivity-saturated} it suffices to show that every irreducible signature is connected.

\subsection{Upright trees and sections}
\label{sec:upright}

Upright trees are a natural family of spanning trees of $Q_n$ that are easily understood. 

\begin{definition}[Tuffley~\cite{tuffley-2012}]
Root all spanning trees of $Q_n$ at $\emptyset$. A spanning tree $T$ of $Q_n$ is \textbf{upright} if for each vertex $X$ of $Q_n$ the path in $T$ from $X$ to the root has length $|X|$. 
\end{definition}

Equivalently, $T$ is upright if for every vertex $X$ of $T$, the first vertex $Y$ on the path in $T$ from $X$ to the root satisfies $Y\subseteq X$. Let $Y=X-\{i\}$, and set $\psi_T(X)=i$. Then $\psi_T$ defines a function $\powernon{n}\to[n]$ such that $\psi_T(X)\in X$ for all $X\in\powernon{n}$. We call such a function a \emph{section} of $\powernon{n}$:

\begin{definition}[Tuffley~\cite{tuffley-2012}]
A function $\psi:\powernon{n}\to[n]$ such that $\psi(X)\in X$ for all $X$ is a \textbf{section} of $\powernon{n}$. If $\psi$ is a section then the signature of $\psi$ is the $n$--tuple 
\[
\sig(\psi) = (a_1,\dots,a_n)
\]
such that $a_i=|\{X:\psi(X)=i\}|$ for all $i$. 
\end{definition}

It is clear that upright trees are equivalent to sections:
\begin{theorem}[Tuffley~{\cite[Lemma~11]{tuffley-2012}} for $n=3$, and Al Fran~{\cite[Lemma~2.2.27]{alfran-2017}} for arbitrary $n$]
The correspondence $T\leftrightarrow\psi_T$ is a bijection between the set of upright spanning trees of $Q_n$ and the set of sections of $\powernon{n}$. Moreover
$\sig(T) = \sig(\psi_T)$ for all $T$.
\end{theorem}

\section{Characterisation of signatures of spanning trees of $Q_n$}
\label{sec:characterisation}

In this section we use Hall's Theorem to prove the following characterisation of the $n$--tuples 
 $\mathcal{S}=(a_1, a_2, \dots, a_n)$ that are the signature of a spanning tree of $Q_n$. 

\begin{theorem}
\label{thm:characterisation}
Let $\mathcal{S}=(a_1,a_2,\dots,a_n)$, where $1\leq a_1\leq a_2 \leq \dots \leq a_n\leq 2^{n-1}$ and $\sum_{i=1}^n a_i=2^n-1$. Then $\mathcal{S}$ is the signature of a spanning tree of $Q_n$ if and only if $\sum_{j=1}^k a_j\geq2^k-1$, for all $k \leq n$. 
\end{theorem}

\begin{remark}
\label{rem:characterisation}
Since $\sum_{i=1}^n a_i=2^n-1$, the signature condition of Theorem~\ref{thm:characterisation} is equivalent to
\[
\sum_{j=k+1}^n a_j \leq 2^n-2^k = 2^k(2^{n-k}-1)
\]
for all $1\leq k\leq n$. 
\end{remark}

\begin{example}[Signatures of $Q_4$]
\label{ex:q4signatures}
Applying Theorem~\ref{thm:characterisation} with $n=4$ we find that there are 18 ordered signatures of $Q_4$:
\begin{center}
\begin{tabular}{cccccc} 
(1, 2, 4, 8)&(1, 2, 5, 7)&(1, 3, 5, 6)&(2, 2, 4, 7)&(2, 3, 4, 6)&(3, 3, 3, 6)\\
(1, 3, 3, 8)&(1, 2, 6, 6)&(1, 4, 4, 6)&(2, 2, 5, 6)&(2, 3, 5, 5)&(3, 3, 4, 5)\\
(2, 2, 3, 8)&(1, 3, 4, 7)&(1, 4, 5, 5)&(2, 3, 3, 7)&(2, 4, 4, 5)&(3, 4, 4, 4)
\end{tabular}
\end{center}
We will discuss the classification of these signatures in Example~\ref{ex:q4signatures-continued}, and the reason for organising them in this way will become apparent then.
\end{example}

The first step in the proof of Theorem~\ref{thm:characterisation} is to reduce it to the problem of characterising signatures of sections of $\powernon{n}$:
\begin{lemma}
\label{lem:signatures-of-sections}
The $n$--tuple $\mathcal{S}=(a_1,a_2,\dots,a_n)$ is the signature of a spanning tree of $Q_n$ if and only if it is the signature of a section of $\powernon{n}$.
\end{lemma}

We give two independent proofs of this fact: one using Martin and Reiner's weighted count~\cite{martin-reiner-2003} of spanning trees of $Q_n$, and the second via edge slides and upright trees.

\begin{proof}[Proof 1 of Lemma~\ref{lem:signatures-of-sections}, via Martin and Reiner's weighted count]
By Martin and Reiner~\cite{martin-reiner-2003} we have
\[
\sum_{T\in\tree{Q_n}} q^{\dir{T}}x^{\dd{T}}
    =q_1\cdots q_n \prod_{S\in\ptwo{n}}
                \sum_{i\in S} q_i (x_i^{-1}+x_i),
\]
in which
\[
q^{\dir{T}} = q_1^{a_1}q_2^{a_2}\dots q_n^{a_n}
\qquad\Leftrightarrow\qquad
\sig(T) = (a_1,a_2,\dots,a_n). 
\]
Set $x_i=1$ for all $i$ to get
\begin{align*}
\sum_{T\in\tree{Q_n}} q^{\dir{T}}
    &=q_1\cdots q_n \prod_{S\in\ptwo{n}}\sum_{i\in S} 2q_i \\
    &= 2^{2^n-n-1}\prod_{S\in\powernon{n}}\sum_{i\in S} q_i.
\end{align*}
Each term in the expansion corresponds to a choice of $i\in S$ for each nonempty subset $S$ of $[n]$, and hence to a section of $\powernon{n}$. 
\end{proof}

\begin{proof}[Proof 2 of Lemma~\ref{lem:signatures-of-sections}, via edge slides and upright trees]
By Tuffley~\cite[Cor.~15]{tuffley-2012}, each spanning tree of $Q_n$ is connected to an upright spanning tree by a sequence of edge slides. The signature is invariant under edge slides, so we conclude that $\mathcal{S}$ is the signature of a spanning tree if and only if it is the signature of an upright tree. But upright spanning trees are equivalent to sections of $\powernon{n}$, and the equivalence is signature-preserving.
\end{proof}

Recall that Hall's Theorem may be stated as follows (see for example~\cite[Thm~11.13]{bona-2017}):
\begin{theorem}[Hall~\cite{hall-1935}]
\label{hall}
Let $G=(A,B)$ be a bipartite graph with $|A|=|B|$. Then $G$ has a perfect matching if and only if for all nonempty $Y\subseteq A$ we have $|Y|\leq |N(Y)|$, where $N(Y)\subseteq B$ is the neighbourhood of $Y$ in $G$.
\end{theorem}

If the stronger condition $|Y|< |N(Y)|$ holds for all proper nonempty $Y$, then for any $a\in A$ and $b\in N(A)$ one may show there exists a perfect matching such that $a$ is matched with $b$. We use this idea to prove our results of Section~\ref{sec:upright-irreducible}. 

We now prove Theorem~\ref{thm:characterisation}. The proof is illustrated in Figure~\ref{fig:matchinggraph}.

\begin{figure}[t]
\begin{center}
\setlength{\vspacing}{-2.25cm}
\setlength{\hspacing}{1.5cm}
\renewcommand{\arraystretch}{4}
\begin{tabular}{c}
\begin{tikzpicture}[vertex/.style={circle,draw,minimum size = 2mm,inner sep=0pt},thick]
\node (1)  at (0,0) [vertex]  {};
\node (2)  at (\hspacing,0) [vertex]  {};
\node (3)  at (2*\hspacing,0) [vertex]  {};
\node (12)  at (3*\hspacing,0) [vertex]  {};
\node (13)  at (4*\hspacing,0) [vertex]  {};
\node (23)  at (5*\hspacing,0) [vertex]  {};
\node (123) at (6*\hspacing,0) [vertex] {};
\node [above] at (1) {$\{1\}$};
\node [above] at (2) {${\{2\}}$};
\node [above] at (3) {${\{3\}}$};
\node [above] at (12) {$\{1,2\}$};
\node [above] at (13) {$\{1,3\}$};
\node [above] at (23) {${\{2,3\}}$};
\node [above] at (123) {$\{1,2,3\}$};
%%%%%%%
\node (a11)  at (0,\vspacing) [vertex]  {};
\node (a12)  at (\hspacing,\vspacing) [vertex]  {};
\node (a21)  at (2*\hspacing,\vspacing) [vertex]  {};
\node (a22)  at (3*\hspacing,\vspacing) [vertex]  {};
\node (a31)  at (4*\hspacing,\vspacing) [vertex]  {};
\node (a32)  at (5*\hspacing,\vspacing) [vertex]  {};
\node (a33) at (6*\hspacing,\vspacing) [vertex] {};
\node [below] at (a11) {$1$};
\node [below] at (a12) {$1$};
\node [below] at (a21) {${2}$};
\node [below] at (a22) {${2}$};
\node [below] at (a31) {${3}$};
\node [below] at (a32) {${3}$};
\node [below] at (a33) {${3}$};
%%%%%%%
\node at (-0.6*\hspacing,0) {$A$:};
\node at (-0.6*\hspacing,\vspacing) {$B$:};
%%%%%%%
\foreach \v/\A in {1/{a11,a12},2/{a21,a22},3/{a31,a32,a33},12/{a11,a12,a21,a22},13/{a11,a12,a31,a32,a33},23/{a21,a22,a31,a32,a33},123/{a11,a12,a21,a22,a31,a32,a33}}
    \foreach \a in \A
        \draw (\v) -- (\a);
\end{tikzpicture} \\ 
\begin{tikzpicture}[vertex/.style={circle,draw,minimum size = 2mm,inner sep=0pt},thick]
\node (1)  at (0,0) [vertex]  {};
\node (2)  at (\hspacing,0) [vertex,fill=blue]  {};
\node (3)  at (2*\hspacing,0) [vertex]  {};
\node (12)  at (3*\hspacing,0) [vertex]  {};
\node (13)  at (4*\hspacing,0) [vertex]  {};
\node (23)  at (5*\hspacing,0) [vertex,fill=blue]  {};
\node (123) at (6*\hspacing,0) [vertex] {};
\node [above] at (1) {$\{1\}$};
\node [above] at (2) {${\{2\}}$};
\node [above] at (3) {${\{3\}}$};
\node [above] at (12) {$\{1,2\}$};
\node [above] at (13) {$\{1,3\}$};
\node [above] at (23) {${\{2,3\}}$};
\node [above] at (123) {$\{1,2,3\}$};
%%%%%%%
\node (a11)  at (0,\vspacing) [vertex]  {};
\node (a12)  at (\hspacing,\vspacing) [vertex]  {};
\node (a21)  at (2*\hspacing,\vspacing) [vertex,fill=red]  {};
\node (a22)  at (3*\hspacing,\vspacing) [vertex,fill=red]  {};
\node (a31)  at (4*\hspacing,\vspacing) [vertex,fill=red]  {};
\node (a32)  at (5*\hspacing,\vspacing) [vertex,fill=red]  {};
\node (a33)  at (6*\hspacing,\vspacing) [vertex,fill=red] {};
\node [below] at (a11) {$1$};
\node [below] at (a12) {$1$};
\node [below] at (a21) {${2}$};
\node [below] at (a22) {${2}$};
\node [below] at (a31) {${3}$};
\node [below] at (a32) {${3}$};
\node [below] at (a33) {${3}$};
%%%%%%%
\node at (-0.6*\hspacing,0) {$A$:};
\node at (-0.6*\hspacing,\vspacing) {$B$:};
%%%%%%%
\foreach \v/\A in {1/{a11,a12},2/{a21,a22},3/{a31,a32,a33},12/{a11,a12,a21,a22},13/{a11,a12,a31,a32,a33},23/{a21,a22,a31,a32,a33},123/{a11,a12,a21,a22,a31,a32,a33}}
    \foreach \a in \A
        \draw (\v) -- (\a);
\foreach \v/\A in {2/{a21,a22},23/{a21,a22,a31,a32,a33}}
    \foreach \a in \A
        \draw [ultra thick] (\v) -- (\a);
\end{tikzpicture}
\end{tabular}
\caption{Illustrating the proof of Theorem~\ref{thm:characterisation}, in the case $\mathcal{S}=(2,2,3)$. 
\emph{Upper figure:} The matching graph $G_\mathcal{S}$.
\emph{Lower figure:}
Checking the Hall condition $|N(Y)|\geq|Y|$ for $Y=\bigl\{\{2\},\{2,3\}\bigr\}\subseteq A$ (filled vertices in $A$).
We have $\supp(Y)=\{2\}\cup\{2,3\}=\{2,3\}$, so the neighbourhood of $Y$ consists of all vertices in $B$ labelled $2$ or $3$ (filled vertices in $B$). Consequently
$|N(Y)|=\sum_{i\in\supp(Y)} a_i = a_2+a_3$.
Since $\mathcal{S}$ is ordered $|N(Y)|=a_2+a_3\geq a_1+a_2 = 4$.
On the other hand, $Y$ is a nonempty subset of $\mathcal{P}_{\geq 1}(\supp(Y))$, so
$|Y|\leq |\mathcal{P}_{\geq 1}(\supp(Y))|=2^2-1=3$. The Hall condition for $Y=\bigl\{\{2\},\{2,3\}\bigr\}$ therefore follows from the condition $\sum_{i=1}^k a_i \geq 2^k-1$ of Theorem~\ref{thm:characterisation}, with $k=|\supp(Y)|=2$.}
\label{fig:matchinggraph}
\end{center}
\end{figure}
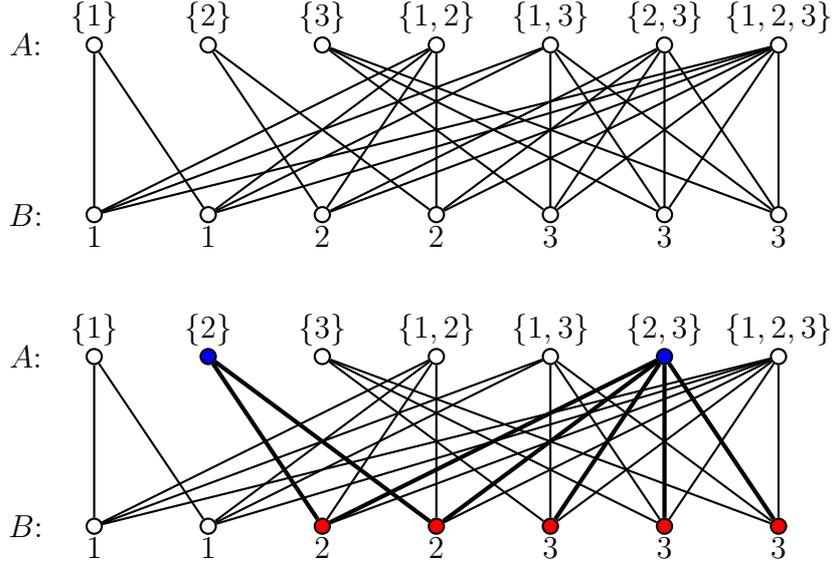

\begin{proof}[Proof of Theorem~\ref{thm:characterisation}]
Let $A$ be $\powernon{n}$, the set of $2^n-1$ nonempty vertices of $Q_n$, and let $B$ be a set of $2^n-1$ vertices of which $a_i$ are labelled $i$, for each $i\in [n]$. For each vertex $V$ in $A$ and $i\in V$ we draw an edge to every vertex in $B$ labelled $i$, as shown in Figure~\ref{fig:matchinggraph} for the case $\mathcal{S}=(2,2,3)$. Let $G_{\mathcal{S}}$ be the resulting bipartite graph with bipartition $(A, B)$.  A section of $\powernon{n}$ with signature $\mathcal{S}$ corresponds to a perfect matching in $G_{\mathcal{S}}$, so we show there is a perfect matching in $G_{\mathcal{S}}$ if and only if the signature condition $\sum_{j=1}^k a_j\geq2^k-1$ is satisfied for all $k \leq n$. 

Given a nonempty subset $Y$ of $A$, define the \emph{support} of $Y$ to be the set
\[
\supp(Y)=\bigcup_{V\in Y} V.
\]
Suppose that $\supp(Y)=\{i_1, i_2\dots, i_k\}$, 
where $1\leq i_1<\dots < i_k\leq n$. Then $i_j\geq j$ for $1\leq j\leq k$, which implies $a_{i_j}\geq a_j$ because $\mathcal{S}$ is ordered.
It follows that the neighbourhood $N(Y)$ of $Y$ in $G_\mathcal{S}$ satisfies
\begin{align*}
|N(Y)|= \sum_{i\in \supp(Y)} a_i
             &=\sum_{j=1}^k a_{i_j}\\
             &\geq \sum_{j=1}^k a_j 
\end{align*}
with equality if $\supp(Y)=\{1, \dots, k\}$. Also 
\[
|Y|\leq |\mathcal{P}_{\geq 1}(\supp(Y))|= 2^k-1, 
\]
with equality if and only if $Y=\mathcal{P}_{\geq 1}(\supp(Y))$. Then we conclude that $|N(Y)|\geq |Y|$ for all $Y\subseteq A$ if and only if $\sum_{j=1}^k a_j\geq 2^k-1$ for all $k\leq n$. Thus by Hall's Theorem there exists a perfect matching in $G$ if and only if $\sum_{j=1}^k a_j\geq 2^k-1$ for all $k$.  
\end{proof}

We conclude this section by proving a lower bound on the growth of an ordered signature.
\begin{lemma}
\label{lem:signature-growth}
Let $\mathcal{S}=(a_1, \dots, a_n)$ be an ordered signature of $Q_n$. Then $i\leq a_i$ for all  $i\in [n]$. 
\end{lemma}

\begin{proof}
We use the fact easily proved by induction that $m(m-1)< 2^m-1$ for all $m$. Let $j< i$. Since $\mathcal{S}$ is ordered we have $a_j\leq a_i$, and therefore $2^i-1\leq \sum_{j=1}^i a_j\leq i a_i$. 

Suppose that $i> a_i$. Then $a_i\leq i-1$ and so $2^i-1\leq i (i-1)$, contradicting the fact that $i(i-1)< 2^i-1$. Therefore $i> a_i$ is impossible, so $i\leq a_i$.
\end{proof}

\section{Classification of signatures of spanning trees of $Q_n$}
\label{sec:classification}
We classify signatures of $Q_n$ as reducible or irreducible as follows. 

\begin{definition}
\label{defn:classification}
Let $\mathcal{S}= (a_1,\dots,a_n)$ be a signature of a spanning tree of $Q_n$. Then $\mathcal{S}$ is  \textbf{reducible} if there exists a proper nonempty subset $R$ of $[n]$ such that $\sum_{i\in R} a_i=2^{|R|}-1$. We say that $R$ is a \textbf{reducing set} for $\mathcal{S}$, and that $\mathcal{S}$ \textbf{reduces over $R$}. If no such set exists then $\mathcal{S}$ is \textbf{irreducible}.

By extension, we will say that a spanning tree $T$ is reducible or irreducible according to whether $\sig(T)$ is reducible or irreducible. If $\sig(T)$ is reducible with reducing set $R$, we will say that \textbf{$T$ reduces over $R$}. 
\end{definition}
Note that if $\mathcal{S}$ is irreducible then $a_i\geq 2$ for all $i$, because if $a_i=1$ then $\mathcal{S}$ reduces over $\{i\}$.

\begin{remark}
\label{rem:classification}
If $\mathcal{S}$ is ordered and $R\subseteq [n]$ satisfies $|R|=r$ then
\[
\sum_{i\in R} a_i \geq \sum_{i=1}^r a_i. 
\]
It follows that an ordered signature has a reducing set of size $r$ if and only if $[r]$ itself is a reducing set. If this holds then we have
$\sum_{i=1}^{r} a_i = 2^r-1$, and moreover $\sum_{i=1}^{k} a_i \geq 2^k-1$ for $1\leq k\leq r$, by the signature condition for $\mathcal{S}$. It follows that $\mathcal{S}'=(a_1,\ldots,a_r)$ is a signature of $Q_r$. Thus, an ordered signature is reducible if and only if it has a initial segment that is a signature of a lower dimensional cube.
More generally, a not-necessarily ordered signature $\mathcal{S}$ is reducible if and only if there is a proper nonempty subset $R$ of $[n]$ such that the restriction of $\mathcal{S}$ to the indices in $R$ gives a signature of $Q_R$.

\end{remark}

\begin{example}
\label{ex:classification}
Consider the following ordered signatures of $Q_7$:
\begin{align*}
\mathcal{S}_1 &= (2,2,4,8,16,32,63), & \mathcal{S}_3 &= (2,2,4,8,15,32,64),\\
\mathcal{S}_2 &= (2,2,3,9,15,33,63), & \mathcal{S}_4 &= (2,2,3,9,15,32,64).
\end{align*}
The signature $\mathcal{S}_1$ is irreducible, and the rest are reducible. Signature $\mathcal{S}_2$ reduces over $[3]$ and $[5]$; signature $\mathcal{S}_3$ reduces over $[5]$ and $[6]$; and signature $\mathcal{S}_4$ reduces over $[3]$, $[5]$ and $[6]$.
\end{example}

\begin{definition}
\label{defn:excess}
Let $\mathcal{S}=(a_1, \dots, a_n)$ be a signature of $Q_n$ and let $1\leq k\leq n$. We define the \textbf{excess of $\mathcal{S}$ at $k$}, $\excess{k}{\mathcal{S}}$, to be
\[
\excess{k}{\mathcal{S}}= \min_{\stacking{K\subseteq [n]}{|K|=k}} \left(\sum_{i\in K} a_i\right) - (2^k-1). 
\]
Thus, the excess at $k$ is the minimum quantity by which a set of $k$ directions exceeds the matching condition of Hall's Theorem. 
Consequently, $\mathcal{S}$ is irreducible if and only if $\excess{k}{\mathcal{S}}\geq 1$ for all $k\leq n-1$, and is reducible if and only if $\excess{k}{\mathcal{S}}=0$ for some $k\leq n-1$. 
Note that by definition $\excess{n}{\mathcal{S}}=0$, and if $\mathcal{S}$ is ordered then the excess at $k$ is simply given by
\[
\excess{k}{\mathcal{S}}=\left(\sum_{i=1}^k a_i\right) - (2^k-1).
\]
\end{definition}

\begin{remark}
\label{rem:excessequivalences}
Observe that for an ordered signature $\mathcal{S}=(a_1,\dots,a_n)$ of $Q_n$ and $r<n$, the following statements are equivalent:
\begin{enumerate}
\item
$(a_1,\dots,a_r)$ is a signature of $Q_r$.
\item
$\mathcal{S}$ reduces over $[r]$.
\item
$\excess{r}{\mathcal{S}}=0$.
\item
$\sum_{i=1}^r a_i = 2^r-1$. 
\end{enumerate}
Note further that if $\excess{k-1}{\mathcal{S}}=\excess{k}{\mathcal{S}}=0$, then $a_k=2^{k-1}$. 
\end{remark}

Reducible signatures of $Q_n$ can be divided into two types: \emph{strictly reducible} and \emph{quasi-irreducible} signatures. In order to define these we first introduce the notion of \emph{saturated} and \emph{unsaturated} signatures as follows.  

\begin{definition}
\label{defn:saturated}
Let $\mathcal{S}=(a_1, \dots, a_n)$ be a signature of $Q_n$. If there exists $r<n$ such that $\excess{k}{\mathcal{S}}=0$ for all $r\leq k\leq n$, then $\mathcal{S}$ is a \textbf{saturated} signature. If no such index exists than $\mathcal{S}$ is \textbf{unsaturated}. 
Equivalently, $\mathcal{S}$ is saturated if and only if it reduces over a set of size $n-1$. 

If $\mathcal{S}$ is ordered and $\excess{k}{\mathcal{S}}=0$ for all $r\leq k\leq n$, then we further say that $\mathcal{S}$ is  \textbf{saturated above direction $r$}. 
\end{definition}

Note that a saturated signature is necessarily reducible.
If the ordered signature $\mathcal{S}$ is saturated above direction $r$ then by Remark~\ref{rem:excessequivalences} we have $a_k=2^{k-1}$ for $r+1\leq k\leq n$,  and moreover the $k$--tuple $(a_1,\ldots,a_k)$ is a signature of $Q_k$ for $r\leq k\leq n$. We may therefore make the following definition:

\begin{definition}
\label{defn:unsaturatedpart}
Let $\mathcal{S}=(a_1,\dots,a_n)$ be an ordered signature of $Q_n$, and 
let $1\leq s\leq n$ be the least index such that $\excess{k}{\mathcal{S}}=0$ for all $s\leq k\leq n$  (such an $s$ exists because $\excess{n}{\mathcal{S}}=0$). Then the $s$--tuple $\unsat(\mathcal{S})$ defined by
\[
\unsat(\mathcal{S})=(a_1,\dots,a_s)
\]
is necessarily an unsaturated signature of $Q_s$, and is the \textbf{unsaturated part} of $\mathcal{S}$. 

If $\mathcal{S}$ is not ordered
we define $\unsat(\mathcal{S})$ to be the restriction of $\mathcal{S}$ to the entries appearing in the unsaturated part of an ordered permutation $\mathcal{S}'$ of $\mathcal{S}$. Write $\mathcal{S}'=(a_1',\ldots,a_n')$, and suppose that $\unsat(\mathcal{S}')= (a_1',\ldots,a_s')$. Then
\[
\mathcal{S}'=(a_1',\ldots,a_s',2^s,2^{s+1},\dots,2^{n-1}),
\]
and $a_i'<2^{s-1}$ for $1\leq i\leq s$. Thus, $\unsat(\mathcal{S})$ is the restriction of $\mathcal{S}$ to the entries satisfying $a_i<2^{s-1}$. Moreover, while there may be more than one permutation of $[n]$ that puts  $\mathcal{S}$ in increasing order (where there are indices $i\neq j$ such that $a_i=a_j$), there is no ambiguity in which indices occur in the unsaturated part.
\end{definition}

We use the unsaturated part to divide reducible signatures into quasi-irreducible and strictly reducible signatures:

\begin{definition}
\label{defn:quasi-strictly}
Let $\mathcal{S}$ be a reducible signature of $Q_n$. Then $\mathcal{S}$ is \textbf{quasi-irreducible} if the unsaturated part $\unsat(\mathcal{S})$ is irreducible. Otherwise, $\mathcal{S}$ is \textbf{strictly reducible}.

By extension, we will say that a reducible spanning tree $T$ of $Q_n$ is quasi-irreducible or strictly reducible according to whether $\sig(T)$ is quasi-irreducible or strictly reducible. 
\end{definition}

\begin{example}
For the signatures appearing in Example~\ref{ex:classification} we have
\begin{align*}
\unsat(\mathcal{S}_1) &= \mathcal{S}_1, & 
\unsat(\mathcal{S}_3) &= (2,2,4,8,15),\\
\unsat(\mathcal{S}_2) &= \mathcal{S}_2, & 
\unsat(\mathcal{S}_4) &= (2,2,3,9,15).
\end{align*}
Signatures $\mathcal{S}_1$ and $\mathcal{S}_2$ are unsaturated, while $\mathcal{S}_3$ and $\mathcal{S}_4$ are both saturated above direction 5. Signatures $\mathcal{S}_2$ and $\mathcal{S}_4$ are strictly reducible (their unsaturated parts both have $[3]$ as a reducing set), and signature $\mathcal{S}_3$ is quasi-irreducible. 
\end{example}

\begin{example}[Classification of signatures in low dimensions]
The unique signature $(1)$ of $Q_1$ is irreducible. Up to permutation $Q_2$ has the unique signature $(1,2)$, which is reducible and saturated, with unsaturated part $(1)$. It is therefore quasi-irreducible. The signatures of $Q_3$ up to permutation are $(1,2,4)$, $(1,3,3)$ and $(2,2,3)$, which are respectively quasi-irreducible, strictly reducible and irreducible. Of these only $(1,2,4)$ is saturated. 
\end{example}

\begin{example}[Classification of signatures of $Q_4$]
\label{ex:q4signatures-continued}
Consider again the signatures of $Q_4$ from Example~\ref{ex:q4signatures}:
\begin{center}
\begin{tabular}{cccccc} 
(1, 2, 4, 8)&(1, 2, 5, 7)&(1, 3, 5, 6)&(2, 2, 4, 7)&(2, 3, 4, 6)&(3, 3, 3, 6)\\
(1, 3, 3, 8)&(1, 2, 6, 6)&(1, 4, 4, 6)&(2, 2, 5, 6)&(2, 3, 5, 5)&(3, 3, 4, 5)\\
(2, 2, 3, 8)&(1, 3, 4, 7)&(1, 4, 5, 5)&(2, 3, 3, 7)&(2, 4, 4, 5)&(3, 4, 4, 4)
\end{tabular}
\end{center}
The signatures in the first column all have $[3]$ as a reducing set, while those in the second and third columns all have $[1]$ as a reducing set. Thus these nine signatures are reducible. The nine signatures appearing in the last three columns are all irreducible.

The signatures in the first column are obtained by appending $2^3=8$ to a signature of $Q_3$ (equivalently, have $[3]$ as a reducing set), so are saturated. The remaining signatures are unsaturated. The reducible signatures in the second and third columns are therefore strictly reducible. For the saturated signatures, we have
\begin{align*}
\unsat(1,2,4,8) &= (1), \\
\unsat(1,3,3,8) &= (1,3,3), \\
\unsat(2,2,3,8) &= (2,2,3),
\end{align*}
so just $(1,3,3,8)$ is strictly reducible, and the other two are quasi-irreducible.
\end{example}

The signatures $(1)$, $(1,2)$, $(1,2,4)$ and $(1,2,4,8)$ seen above are the first four members of an infinite family of signatures:

\begin{definition}
For $n\geq 1$ let $\supersat$ be the $n$--tuple $(a_1,\ldots,a_n)$ defined by
$a_i=2^{i-1}$ for $1\leq i\leq n$:
\[
\supersat = (1,2,4,8,\dots,2^{n-1}). 
\]
 Then $\sum_{i=1}^k a_i=2^k-1$ for all $1\leq k\leq n$, so $\supersat$ is a signature of $Q_n$.
\end{definition}

Observe that $\supersat$ satisfies $\excess{k}{\supersat}=0$ for $1\leq k\leq n$. It follows that \supersat\ is saturated above direction 1 for all $n\geq 2$, and $\unsat(\supersat)=(1)$ for all $n$. 
For $n\geq 2$ we will say that \supersat\ is \emph{supersaturated}:

\begin{definition}
Let $n\geq 2$. A signature $\mathcal{S}=(a_1, \dots, a_n)$ is \textbf{supersaturated} if $\excess{k}{\mathcal{S}}=0$ for all $1\leq k\leq n$. Equivalently, $\mathcal{S}$ is supersaturated if and only if it is a permutation of \supersat.
\end{definition}

\section{Consequences of the classification for upright trees}
\label{sec:uprighttrees}

In this section and the next we show that reducibility places strict structural constraints on a spanning tree of $Q_n$. We 
begin by restricting our attention to upright spanning trees, which are easily understood through their equivalence with sections of $\powernon{n}$. 

By identifying each upright tree $T$ with its associated section $\psi_T$, we may regard an upright tree as a choice of $x\in X$ at each nonempty subset $X$ of $[n]$. 
We may ask the following question:
\begin{question}
\label{q:existenceofsection}
Given a signature $\mathcal{S}$ of $Q_n$, a nonempty subset $X$ of $[n]$, and an element $x$ of $X$, does there exist an upright spanning tree $T$ with signature $\mathcal{S}$ such that $\psi_T(X)=x$?
\end{question}
Lemma~\ref{lem:uprightreducible} shows that, if $\mathcal{S}$ is reducible, then it is always possible to choose a nonempty subset $X$ of $[n]$ and an element $x$ of $X$ such that the answer to this question is ``no''. In contrast, Corollary~\ref{cor:specify1} shows that for irreducible $\mathcal{S}$, the answer to this question is always ``yes'', regardless of the choice of nonempty $X\subseteq [n]$ and $x\in X$. Loosely speaking, this means that we may arbitrarily specify the value of a section with irreducible signature $\mathcal{S}$ at any single vertex of our choice. We further show that under certain conditions (typically expressed in terms of the excess) we can specify the value of a section with signature $\mathcal{S}$ at one or more additional vertices.

\subsection{Reducible upright trees}
\label{sec:upright-reducible}

We show that reducibility constrains the edges of an upright spanning tree:

\begin{lemma}
\label{lem:uprightreducible}
Let $\mathcal{S}=(a_1, \dots, a_n)$ be a reducible signature of $Q_n$, and let $R$ be a reducing set for $\mathcal{S}$. Let $T$ be an upright spanning tree of $Q_n$ with signature $\mathcal{S}$ and let $X$ be a nonempty vertex of $Q_n$. Then $\psi_T(X)\in R$ if and only if $X\subseteq R$. 
\end{lemma}

This answers Question~\ref{q:existenceofsection} for reducible signatures, by showing that if $x\in[n]$ is chosen such that $x\in R$, then the answer is ``yes'' only if $X\subseteq R$.

\begin{proof}
The fact that $\psi_T(X)\in R$ for $X\subseteq R$ is immediate from the fact that $\psi_T$ is a section. For the converse, observe that in total $T$ has $\sum_{i\in R} a_i= 2^{|R|}-1$ edges in directions belonging to $R$, and $R$ has $2^{|R|}-1$ nonempty subsets. Thus all edges of $T$ in directions belonging to $R$ are accounted for at the subsets of $R$, so we must have $\psi_T(X)\notin R$ for $X\nsubseteq R$. 
\end{proof}

Applying Lemma~\ref{lem:uprightreducible} to an ordered saturated signature we get:
\begin{corollary}
\label{cor:upright-saturated}
Let $\mathcal{S}=(a_1, \dots, a_n)$ be an ordered signature. If $\mathcal{S}$ is saturated above direction $r$ and $X\nsubseteq [r]$, then $\psi_T(X)=\max X$. 
\end{corollary}
\begin{proof}
Since $\mathcal{S}$ is saturated above direction $r$, it reduces over $[s-1]$ for each $s>r$. If $\max X=s$, then $X\subseteq [s]$ but  $X\nsubseteq [s-1]$. Therefore $\psi_T(X)$ belongs to $[s]$ but not $[s-1]$, and hence $\psi_T(X)=s=\max X$.
\end{proof}

\begin{corollary}
\label{cor:saturatedcount}
Let the ordered  signature $\mathcal{S}=(a_1, \dots, a_n)$ of $Q_n$ be saturated above direction $r$. Then the number of upright spanning trees of $Q_n$ with signature $\mathcal{S}$ is equal to the number of upright spanning trees of $Q_r$ with signature $\mathcal{S}'=(a_1, \dots, a_r)$.

In particular, if the unsaturated part of $\mathcal{S}$ consists of the first $s$ entries of $\mathcal{S}$, then the number of upright spanning trees of $Q_n$ with signature $\mathcal{S}$ is equal to the number of upright spanning trees of $Q_s$ with signature $\unsat(\mathcal{S})$.
\end{corollary}

\begin{proof}
Given an upright spanning tree $T$ of $Q_n$ with signature $\mathcal{S}$
let $T'=T\cap Q_r$. Then $T'$ is an upright spanning tree of $Q_r$ with associated section 
$\psi_{T'}=\psi_T\big|_{\powernon{r}}$, the restriction of $\psi_T$ to $\powernon{r}$. 
Since $\mathcal{S}$ reduces over $[r]$ Lemma~\ref{lem:uprightreducible}
implies $\psi_T(X)\in [r]$ if and only if $X\subseteq [r]$, and it follows that $\sig(T')=\mathcal{S}'$. 

Conversely, given an upright spanning tree $T'$ of $Q_r$ with signature $\mathcal{S}'$, we can extend $T'$ to an upright spanning tree $T$ of $Q_n$ 
such that $T'=T\cap Q_r$
by defining
\[
\psi_T(X) = \begin{cases}
            \psi_{T'}(X), & \text{if $X\subseteq R$}, \\
            \max X,     & \text{otherwise}.
            \end{cases}
\]
For each $1\leq k\leq n$ there are $2^{k-1}$ subsets $X$ of $[n]$ such that $\max X=k$, so 
 $\sig(T)=(a_1,\ldots,a_r,2^r,2^{r+1},\ldots,2^{n-1})=\mathcal{S}$. Moreover, Corollary~\ref{cor:upright-saturated} shows that any upright spanning tree of $Q_n$ with signature $\mathcal{S}$ that extends $T'$ must co-incide with $T$. It follows that the map $T\mapsto T\cap Q_r$ is a bijection from the set of upright spanning trees of $Q_n$ with signature $\mathcal{S}$ to the set of upright spanning trees of $Q_r$ with signature $\mathcal{S}'$, proving the result.
\end{proof}

\begin{corollary}
There is only one upright spanning tree of $Q_n$ with the supersaturated signature $\supersat=(1, 2, 4, 8, \dots, 2^{n-1})$. 
\end{corollary}

\begin{proof}
The signature $\supersat$ satisfies $\unsat(\supersat)=(1)$. The signature $(1)$ has a unique upright tree, so the result follows immediately by Corollary~\ref{cor:saturatedcount}.
\end{proof}

\subsection{Irreducible upright trees}
\label{sec:upright-irreducible}

We now consider irreducible upright spanning trees, and show that in contrast to Lemma~\ref{lem:uprightreducible}, for $\mathcal{S}$ irreducible
the answer to Question~\ref{q:existenceofsection} is always ``yes'':
given nonempty $X\subseteq [n]$ and $x\in X$, there always exists an upright spanning tree $T$ with signature $\mathcal{S}$ such that $\psi_T(X)=x$. Since irreducible signatures satisfy $\excess{k}{\mathcal{S}}\geq 1$ for all $k<n$, we deduce this as a corollary to Theorem~\ref{thm:excess-ell}, which loosely speaking says that if $\excess{k}{\mathcal{S}}\geq \ell$ for all $k<n$, then we may arbitrarily specify the value of a section at $\ell$ vertices. In fact, the condition $\excess{k}{\mathcal{S}}\geq \ell$ for all $k<n$ appears to be a little stronger than necessary. For $\ell=2$ we show in Theorem~\ref{thm:specifytwo} that, under certain conditions, we can specify the value of a section at two vertices even when we do not have $\varepsilon^{\mathcal{S}}_k\geq 2$ for all $k$. To prove this result we require Lemma~\ref{lem:a_i=a_{i+1}}, which shows that when $a_k$ and $a_{k+1}$ are close enough, the excess at $k$ must be at least 2.

\begin{theorem}
\label{thm:excess-ell}
Let $\ell$ be a positive integer, and let $\mathcal{S}$ be a signature of $Q_n$ such that $\excess{k}{\mathcal{S}}\geq\ell$ for $1\leq k<n$. Let $X_1,\dots,X_\ell$ be distinct nonempty vertices of $Q_n$, and let $x_t\in X_t$ for each $t$. Then there is an upright spanning tree $T$ of $Q_n$ with signature $\mathcal{S}$ such $\psi_T(X_t)=x_t$ for $1\leq t\leq\ell$.
\end{theorem}

\begin{proof}
Let $G_{\mathcal{S}}$ be the matching graph with bipartition $(A, B)$ constructed in the proof of Theorem~\ref{thm:characterisation}. By hypothesis we have
\[
\excess{1}{\mathcal{S}} = \min_{i\in[n]} a_i -1 \geq \ell,
\]
so $a_i\geq \ell+1$ for all $i$. It follows that there exists a partial matching $M$ in $G_\mathcal{S}$ such that $X_t$ is matched with a vertex $v_t\in B$ labelled $x_t$ for $1\leq t\leq\ell$. 
Let $G_{\mathcal{S}}'$ be the matching graph with the vertices $X_1,\dots,X_\ell$, $v_1,\dots,v_\ell$ and all incident edges deleted.
We show that $M$ can be extended to a perfect matching in $G_\mathcal{S}$ by showing that there exists a perfect matching in $G_{\mathcal{S}}'$. 

Let $A'= A-\{X_1,\dots,X_\ell\}$, $B'=B-\{v_1,\ldots,v_\ell\}$, and for $1\leq i\leq n$ let $a_i'$ be the number of vertices labelled $i$ in $B'$.  Given $\emptyset\neq Y\subseteq A'$ let $Z=\supp(Y)$, the union of the sets in $Y$.
If $Z=[n]$ then $N(Y)=B'$, and since $|Y|\leq |A'|=|B'|$ the Hall condition holds for $Y$. Otherwise we have $|Y|\leq 2^{|Z|}-1$ and $\excess{|Z|}{\mathcal{S}}\geq\ell$, so 
\[
|N(Y)| = \sum_{i\in Z} a_i' \geq \left(\sum_{i\in Z}a_i\right)-\ell
  \geq (2^{|Z|}+\excess{|Z|}{\mathcal{S}}-1)-\ell \geq 2^{|Z|}-1\geq |Y|.
\]
Therefore the Hall condition holds for all nonempty $Y\subseteq A'$, so $G_\mathcal{S}'$ has a perfect matching, as required. The resulting perfect matching in $G_\mathcal{S}$ extending $M$ corresponds to a section $\psi$ of $\powernon{n}$ such that $\psi(X_t)=x_t$ for $1\leq t\leq\ell$, proving the existence of the required upright spanning tree.
\end{proof}

Specialising to the case $\ell=1$ we answer Question~\ref{q:existenceofsection} in the affirmative for irreducible signatures:

\begin{corollary}
\label{cor:specify1}
Let $\mathcal{I}$ be an irreducible signature of $Q_n$. Let $X$ be a nonempty vertex of $Q_n$ and let $x\in X$. Then there exists an upright spanning tree $T$ of $Q_n$ with signature $\mathcal{I}$ such that $\psi_{T}(X)=x$. 
\end{corollary}

\begin{proof}
Since $\mathcal{I}$ is irreducible it satisfies $\excess{k}{\mathcal{I}}\geq 1$ for $1\leq k<n$. The result therefore follows immediately from Theorem~\ref{thm:excess-ell}.
\end{proof}

We use the following lemma to show for $\ell=2$ that the excess condition of Theorem~\ref{thm:excess-ell} can be weakened slightly under certain conditions.

\begin{lemma}
\label{lem:a_i=a_{i+1}}
Let $n\geq 4$ and let $\mathcal{I}=(a_1, \dots, a_n)$ be an ordered irreducible signature of $Q_n$. Suppose that, for some $i\in \{2, \dots, n-1\}$, we have $a_{i+1}-a_i\leq 1$. Then $\varepsilon_i^{\mathcal{I}}\geq 2$.
\end{lemma}

We note that the condition $n\geq4$ is necessary in Lemma~\ref{lem:a_i=a_{i+1}}. The irreducible signature $(2,2,3)$ with $i=2$ is a counterexample for $n=3$. 

\begin{proof}
Since $\mathcal{I}$ is irreducible we necessarily have $\varepsilon_i^{\mathcal{I}}\geq 1$. Suppose that $\varepsilon_i^{\mathcal{I}}=1$. Then since $\mathcal{I}$ is irreducible we have
\[
2^{i}=\sum_{j=1}^i a_j=\sum_{j=1}^{i-1}a_j+a_i\geq 2^{i-1}+a_i,
\]
and therefore $a_i\leq 2^{i-1}$. If $i<n-1$ then
\[
2^{i+1}\leq \sum_{j=1}^{i+1} a_j =\sum_{j=1}^i a_j + a_{i+1}= 2^{i}+a_{i+1},
\]
which implies $a_{i+1}\geq 2^i$. But then $a_{i+1}-a_i\geq 2^{i-1}\geq 2$, a contradiction. Similarly, if $i=n-1$ then 
\[
2^{i+1}-1\leq \sum_{j=1}^{i+1} a_j =\sum_{j=1}^i a_j + a_{i+1}= 2^{i}+a_{i+1},
\]
which implies $a_{i+1}\geq 2^i-1$. Then $a_{i+1}-a_i\geq 2^{i-1}-1=2^{n-2}-1\geq 3$, and we again reach a contradiction. Therefore it must in fact be the case that $\varepsilon_i^{\mathcal{I}}\geq 2$. 
\end{proof}

For $\ell=2$ we may weaken the excess condition of Theorem~\ref{thm:excess-ell} as follows:

\begin{theorem}
\label{thm:specifytwo}
Let $n\geq 4$, and let $\mathcal{I} = (a_1, \dots, a_n)$ be an ordered irreducible signature of $Q_n$.  Let $X_1,X_2$ be distinct nonempty vertices of $Q_n$, and let $x_t\in X_t$ for $t=1,2$. Suppose that one of the following two conditions holds:
\begin{enumerate}
\item\label{item:excess2}
$\varepsilon_k^{\mathcal{I}}\geq 2$ for all $k\geq \max \{x_1,x_2\}$.
\item\label{item:maxlemma}
$x_1\neq x_2$, and either $x_1=\max X_1$ or $x_2=\max X_2$.
\end{enumerate}
Then there exists an upright spanning tree $T$ of $Q_n$ with signature $\mathcal{I}$ such $\psi_T(X_t)=x_t$ for $t=1,2$.
\end{theorem}

\begin{proof}
Let $G_{\mathcal{I}}$ be the matching graph with bipartition $(A, B)$ constructed in the proof of Theorem~\ref{thm:characterisation}. Since $\mathcal{I}$ is irreducible we have $a_i\geq 2$ for all $i$, so there exists a partial matching $M$ of $A$ into $B$ such that $X_t$ is matched with a vertex $v_t$ labelled $x_t$ for $t=1,2$. 
Let $G_{\mathcal{I}}'$ be the matching graph with the vertices $X_1,X_2$, $v_1,v_2$ and all incident edges deleted.
We show that $M$ can be extended to a perfect matching in $G_\mathcal{I}$ by showing that there exists a perfect matching in $G_{\mathcal{I}}'$. 

Let $A'= A-\{X_1,X_2\}$, $B'=B-\{v_1,v_2\}$, and for $1\leq i\leq n$ let $a_i'$ be the number of vertices labelled $i$ in $B'$.  Given $\emptyset\neq Y\subseteq A'$ let $Z=\supp(Y)$, the union of the sets in $Y$, and set $z=|Z|$. 
If $z=n$ then $N(Y)=B'$, and since $|Y|\leq |A'|=|B'|$ the Hall condition holds for $Y$. Otherwise we have
\begin{equation}
\label{eq:n(y)}
|N(Y)| = \sum_{i\in Z} a_i' = \left(\sum_{i\in Z}a_i\right)-\chi_Z(x_1)-\chi_Z(x_2),
\end{equation}
where $\chi_Z:[n]\to\{0,1\}$ is the characteristic function of $Z$; and
\[
|Y|\leq 2^{z}-1 - \chi_{\power{Z}}(X_1)-\chi_{\power{Z}}(X_2)\leq  2^{z}-1,
\]
where $\chi_{\power{Z}}:\power{[n]}\to \{0,1\}$ is the characteristic function of $\power{Z}$. 
Since $\mathcal{I}$ is irreducible we have $\sum_{i\in Z}a_i \geq 2^{z}$, so
\[
|N(Y)|\geq 2^{z}-2,
\]
with equality possible only if $\sum_{i\in Z}a_i = 2^{z}$ and $x_1,x_2\in Z$. On the other hand we have $|Y|\leq  2^{z}-1$, with equality possible only if $X_1,X_2\nsubseteq Z$ and $Y=\mathcal{P}_{\geq 1}(Z)$. 
It follows that $|N(Y)|\geq|Y|$ except possibly when $x_1,x_2\in Z$ and  $X_1,X_2\nsubseteq Z$. 

Suppose then that $x_1,x_2\in Z$ but $X_1,X_2\nsubseteq Z$. We show under each of the conditions given in the theorem that we have $\sum_{i\in Z} a_i \geq 2^{z}+1$, so that $|N(Y)|\geq 2^{z}-1\geq |Y|$ as needed.

\begin{enumerate}
\item
Suppose that $\excess{k}{\mathcal{I}}\geq 2$ for all $k\geq m=\max\{x_1,x_2\}$.
As in the proof of Theorem~\ref{thm:characterisation} let
$Z=\{i_1, i_2, \dots, i_z\}$, where $i_1<i_2<\dots <i_z$.
If $z\geq m$ then $\excess{z}{\mathcal{I}}\geq 2$, so $\sum_{i\in Z} a_i\geq 2^z+1$ and we are done. Otherwise we have $z<m$, and then 
$i_z>z$, because $m\in Z$ but $|Z|<m$. Therefore $a_{i_z}\geq a_{z+1}$, because $\mathcal{I}$ is ordered. If $\sum_{i\in Z} a_i \geq 2^{z}+1$ does not hold then
\begin{equation}
\label{eq:squeeze}
2^{z}\geq \sum_{i\in Z} a_i 
  =  \sum_{s=1}^z a_{i_s}
 \geq\sum_{s=1}^{z-1} a_s+ a_{i_z} 
 \geq\sum_{s=1}^{z-1} a_s+ a_{z+1} 
 \geq\sum_{s=1}^z a_s\geq 2^z,
\end{equation}
and so $a_{i_z}=a_{z+1}=a_{z}$.

If $z\geq 2$ then by Lemma~\ref{lem:a_i=a_{i+1}} we have $\excess{z}{\mathcal{I}}\geq 2$, and so in fact
\[
\sum_{i\in Z} a_i 
 \geq\sum_{s=1}^z a_s\geq 2^z+1
\]
after all. Otherwise, if $z=1$ then $a_i=a_1$ for $1\leq i\leq i_z$, and so in particular for $1\leq i\leq m$. But then if $a_1=2$ we have $\sum_{i=1}^{m} a_i = 2m < 2^{m}+1$, contradicting our hypothesis that $\excess{m}{\mathcal{I}}\geq 2$. Therefore $\sum_{i\in Z} a_i =a_m \geq 3 = 2^{z}+1$ in this case also.
\item
Under the hypothesis that $x_1=\max X_1$ or $x_2=\max X_2$ we may assume without loss of generality that $x_1=\max X_1$.  
As above we let $Z=\{i_1, i_2, \dots, i_z\}$, where $i_1<i_2<\dots <i_z$, and we note that $z\geq 2$ because $x_1\neq x_2$ and $\{x_1,x_2\}\subseteq Z$. Then since $x_1=\max X_1\in Z$ and $X_1\nsubseteq Z$ it cannot be the case that $Z=[z]$, so as in Case~\ref{item:excess2} we have $i_z>z$ and hence $a_{i_z}\geq a_{z+1}$. Arguing as in Equation~\eqref{eq:squeeze} we therefore get $a_{z+1}=a_z$, and then since $z\geq 2$ we again 
have $\sum_{i\in Z}a_i\geq 2^z+1$, by Lemma~\ref{lem:a_i=a_{i+1}}. 
\end{enumerate}
Therefore the Hall condition holds for all nonempty $Y\subseteq A'$, so $G_\mathcal{I}'$ has a perfect matching, as required. The resulting perfect matching in $G_\mathcal{I}$ extending $M$ corresponds to a section $\psi$ of $\powernon{n}$ such that $\psi(X_t)=x_t$ for $t=1,2$, proving the existence of the required upright spanning tree.
\end{proof}

For completeness we consider the extent to which the hypothesis $x_1\neq x_2$ is necessary in Case~\ref{item:maxlemma} of Theorem~\ref{thm:specifytwo}. We show that this hypothesis can in fact be eliminated except in very limited circumstances:

\begin{proposition}
\label{rem:maxlemma}
Under the hypotheses of Theorem~\ref{thm:specifytwo}, suppose that $x_1=x_2=x$, and either $\max X_1$ or $\max X_2$ is equal to $x$. Then the conclusion of  Theorem~\ref{thm:specifytwo} still holds unless $x=2$, $a_1=a_2=2$, and (perhaps after permuting them) we have $X_1=\{1,2\}$ and $X_2\nsubseteq\{1,2\}$. 
\end{proposition}

\begin{proof}
  In the proof of Case~\ref{item:maxlemma} of Theorem~\ref{thm:specifytwo}, the hypothesis $x_1\neq x_2$ is used only to rule out the possibility $z=1$ when $x_1,x_2\in Z$ but $X_1,X_2\nsubseteq Z$. We therefore check when $|Y|>|N(Y)|$ can hold under these conditions.

Since $z=|Z|=1$ and $x\in Z$ we must have $Z=\{x\}$, which in turn implies $Y=\{\{x\}\}$. From equation~\eqref{eq:n(y)} we have $|N(Y)|=a_x-2$, so if $|Y|>|N(Y)|$ we must have $a_x\leq 2$. 
Irreducibility of $\mathcal{I}$ rules out the possibility $a_x=1$, so we must have $a_x=2$, and then $x\leq 2$ by Lemma~\ref{lem:signature-growth}. We need not consider the case $x=1$: if $x=1$ then the hypothesis $x=\max X_t$ for some $t$ implies either $X_1$ or $X_2$ is equal to $\{1\}$, which means that the vertex $\{1\}$ is already matched by $M$ and does not belong to $G_{\mathcal{I}}'$. So suppose that $x=2$. Then $a_1=a_2=2$ by irreducibility, and the condition $x_t=\max X_t$ for $t=1$ or $2$ together with $X_1,X_2\nsubseteq Z$ implies that (perhaps after relabelling) we have $X_1=\{1,2\}$ and $X_2\nsubseteq\{1,2\}$, as claimed. 
We see moreover that in this case the required matching in $G_\mathcal{I}$ does not in fact exist, because the three distinct vertices $\{2\}$, $X_1=\{1,2\}$ and $X_2$ must all be matched with vertices in $B$ labelled 2, and there are only two such vertices. 
\end{proof}

\section{Structural consequences of reducibility}
\label{sec:structural}

We now turn our attention to the consequences of reducibility for arbitrary spanning trees. We give a structural characterisation of reducible trees in Section~\ref{sec:structural-characterisation}, then use this to
show in Section~\ref{sec:decompose} that a tree that reduces over 
a set of size $r$ decomposes as a sum of $2^r$ spanning trees of $Q_{n-r}$, together with a spanning tree of a certain contraction of $Q_n$ with underlying simple graph $Q_r$.
The constructions required to state this result are defined in Section~\ref{sec:partition-and-quotient}.
We then show in Section~\ref{sec:reducibleslides} that this decomposition is realised by a graph isomorphism between edge slide graphs.
We conclude the section by applying the results to several special cases in Section~\ref{sec:special}.

\subsection{Definitions and notations II}
\label{sec:partition-and-quotient}

We present some further definitions needed for our results in this section.

\subsubsection{Notation}

We will be working with the sets of all spanning trees and all signatures that reduce over a given proper non-empty subset $R$ of $[n]$. 
We therefore introduce the following notation:

\begin{definition}
Given a proper non-empty subset $R$ of $[n]$, we define
\begin{align*}
\rtree{R} &= \{T\in \tree{Q_n} :\text{$T$ reduces over $R$}\},\\
\red{R} &= \{\mathcal{S}\in\Sig(Q_n):\text{$\mathcal{S}$ reduces over $R$}\}.
\end{align*}
\end{definition}

\subsubsection{Partitioning the $n$--cube}
\label{sec:partition}

Given a subset $R\subseteq [n]$, we partition $Q_n$ into $2^{|R|}$ copies of $Q_{n-|R|}$ as follows: 

\begin{definition}
For any $X\subseteq R$ (including the empty set), let $Q_n(R, X)$ be the induced subgraph of $Q_n$ with vertices
\[
V(Q_n(R,X))=\{W\subseteq[n]: W\cap R=X\} = \{X\cup Y: Y\subseteq[n]-R\}.
\]
The cases $R=\{1,3\}$ and $R=\{1\}$ with $n=3$ are illustrated in Figure~\ref{fig:partition}. 
For any subgraph $H$ of $Q_n$ we further define $H(R,X)=H\cap Q_n(R,X)$. Thus $H(R,X)$ is the subgraph of $H$ induced by the vertices $W\in V(H)$ satisfying $W\cap R= X$. 
\end{definition}

Observe that $Q_n(R, X)=(Q_{[n]-R})\oplus X$, and so is an $(n-|R|)$--cube; and if $T$ is a spanning tree of $Q_n$, then $T(R,X)$ is a spanning forest\footnote{We use \emph{spanning forest} in the sense of a spanning subgraph that is a forest, and not in the sense of a maximal spanning forest. That is, we do not require each component of a spanning forest of $G$ to be a spanning tree of the component of $G$ it belongs to.} of $Q_n(R,X)$. Note further that 
\begin{itemize}
\item
every edge of $Q_n$ in a direction $i\notin R$ belongs to $Q_n(R, X)$ for some $X$; and 
\item
every edge of $Q_n$ in a direction $j\in R$ joins a vertex of $Q_n(R, X)$ to the corresponding vertex of $Q_n(R, X\oplus\{j\})$ for some $X$.
\end{itemize}
For any $X_1, X_2\subseteq R$ such that $X_1\neq X_2$ we have 
\[
 Q_n(R, X_1)\cap Q_n(R, X_2)=\emptyset.
 \]

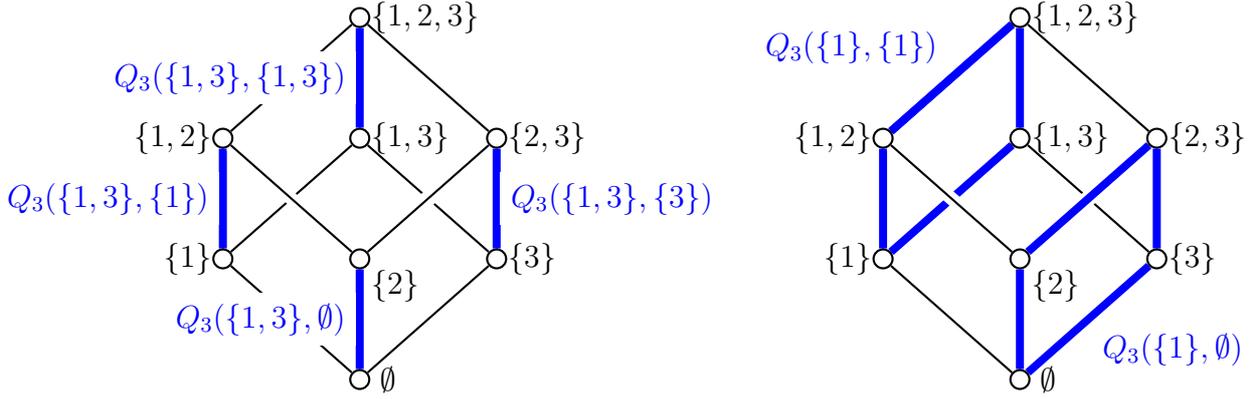
\begin{figure}
\begin{center}
\setlength{\vspacing}{1.6cm}
\setlength{\hspacing}{1.8cm}
\begin{tabular}{cc}
\begin{tikzpicture}[vertex/.style={circle,draw,inner sep=1pt,minimum size=2.5mm},thick]
\node (123) at (0,3*\vspacing) [vertex] {};
\node (12)  at (-\hspacing,2*\vspacing) [vertex]  {};
\node (13)  at (0,2*\vspacing) [vertex]  {};
\node (23)  at (\hspacing,2*\vspacing) [vertex]  {};
\node (1)  at (-\hspacing,\vspacing) [vertex]  {};
\node (2)  at (0,\vspacing) [vertex]  {};
\node (3)  at (\hspacing,\vspacing) [vertex]  {};
\node (0)  at (0,0) [vertex] {};
\foreach \x/\Y in {0/{1,2,3},1/{12,13},2/{12,23},3/{13,23},123/{12,13,23}}
   {\foreach \y in \Y
       \draw (\x) -- (\y);};
\draw [color=white,line width=4] (2)--(12);
\draw  (2)--(12);
\draw [color=white,line width=6pt] (2)--(23);
\draw  (2)--(23);
\node [right] at (0) {$\;\emptyset$};
\node [right] at (3) {$\{3\}$};
\node [right] at (23) {$\{2,3\}$};
\node [right] at (123) {$\{1,2,3\}$};
\node [left] at (12) {$\{1,2\}$};
\node [below right] at (2) {$\{2\}$};
\node [left] at (1) {$\{1\}$};
\node [right] at (13) {$\{1,3\}$};
% Q3({1},{1})
\path (0) edge[line width=3,color=blue] node[midway,left,fill=white] {$Q_3(\{1,3\},\emptyset)$} (2);
\path (1) edge[line width=3,color=blue] node[midway,left,fill=white] {$Q_3(\{1,3\},\{1\})$} (12);
\path (3) edge[line width=3,color=blue] node[midway,right,fill=white] {$Q_3(\{1,3\},\{3\})$} (23);
\path (13) edge[line width=3,color=blue] node[midway,left,fill=white] {$Q_3(\{1,3\},\{1,3\})$} (123);
\draw [color=white,line width=5] (2)--(12);
\draw (12) -- (2);
\end{tikzpicture}& 
\begin{tikzpicture}[vertex/.style={circle,draw,inner sep=1pt,minimum size=2.5mm},thick]
\node (123) at (0,3*\vspacing) [vertex] {};
\node (12)  at (-\hspacing,2*\vspacing) [vertex]  {};
\node (13)  at (0,2*\vspacing) [vertex]  {};
\node (23)  at (\hspacing,2*\vspacing) [vertex]  {};
\node (1)  at (-\hspacing,\vspacing) [vertex]  {};
\node (2)  at (0,\vspacing) [vertex]  {};
\node (3)  at (\hspacing,\vspacing) [vertex]  {};
\node (0)  at (0,0) [vertex] {};
\foreach \x/\Y in {0/{1,2,3},1/{12,13},2/{12,23},3/{13,23},123/{12,13,23}}
   {\foreach \y in \Y
       \draw (\x) -- (\y);};
\draw  (2)--(12);
\draw [color=white,line width=6pt] (2)--(23);
\draw  (2)--(23);
\node [right] at (0) {$\;\emptyset$};
\node [right] at (3) {$\{3\}$};
\node [right] at (23) {$\{2,3\}$};
\node [right] at (123) {$\{1,2,3\}$};
\node [left] at (12) {$\{1,2\}$};
\node [below right] at (2) {$\{2\}$};
\node [left] at (1) {$\{1\}$};
\node [right] at (13) {$\{1,3\}$};
% Q3({1},{1})
\path (12) edge[line width=3,color=blue] node[midway,above left] {$Q_3(\{1\},\{1\})$} (123);
\draw [line width=3,color=blue] (12) -- (123);
\draw [line width=3,color=blue] (13) -- (123);
\draw [line width=3,color=blue] (1) -- (13);
\draw [line width=3,color=blue] (12) -- (1);
% Q3({1},empty)
\draw [line width=3,color=blue] (2) -- (23);
\path (0) edge[line width=3,color=blue] node[midway,below right] {$Q_3(\{1\},\emptyset)$} (3);
\draw [line width=3,color=blue] (0) -- (2);
\draw [line width=3,color=blue] (3) -- (23);
\draw [color=white,line width=5] (2)--(12);
\draw (12) -- (2);
\end{tikzpicture}
\end{tabular}
\caption{The subcubes $Q_3(R,X)$ for $X\subseteq R$ for $R=\{1,3\}$ (left) and $R=\{1\}$ (right). In each case we get $2^{|R|}$ subcubes of dimension $3-|R|$, together containing all edges of $Q_n$ in directions not belonging to $R$.}
\label{fig:partition}
\end{center}
\end{figure}

\subsubsection{Quotienting the $n$--cube}
\label{sec:quotient}

\begin{definition}
\label{def:contraction}
Let $S\subseteq[n]$. We define \contract{S}\ to be the graph obtained from $Q_n$ by contracting every edge in direction $j$, for all $j\in S$. 

In practice we will be most interested in the case where $S=\bar{R}:=[n]-R$, for some $R\subseteq[n]$. 
The contractions \contractb[3]{R}\ for $R=\{1,3\}$ and $R=\{1\}$ are illustrated in Figure~\ref{fig:quotient}. 
For $R\subseteq[n]$ the contraction \contractb{R}\ is the graph obtained from $Q_n$ by contracting every edge in direction $j$, for all $j\notin R$. 
The construction has the effect of contracting each subcube $Q_n(R,X)$ to a single vertex, which we may label $X$, for each $X\subseteq R$. The resulting graph \contractb{R}\ is a multigraph with underlying simple graph $Q_R$, and $2^{n-|R|}$ parallel edges for each edge of $Q_R$: one for each element of $\power{\bar{R}}$. We regard \contractb{R}\ as having vertex set $V(Q_R)=\power{R}$ and edge set $E(Q_R)\times\power{\bar{R}}$, where the edge $(e,Y)\in E(Q_R)\times\power{\bar{R}}$ joins the endpoints of $e$. We define 
\[
\capr{R}:\contractb{R}\to Q_R
\] 
to be the projection from \contractb{R}\ to the underlying simple graph. This map fixes all the vertices and sends $(e,Y)\in E(Q_R)\times\power{\bar{R}}$ to $e\in E(Q_R)$. 
\end{definition}

\begin{figure}
\begin{center}
\begin{tabular}{ccc}
\begin{tikzpicture}[vertex/.style={circle,draw,fill=blue,minimum size = 2mm,inner sep=0pt},thick]
\setlength{\hspacing}{2cm}
\node (0) at (0,0) [vertex,label=below:$\emptyset$] {};
\node (1) at (-\hspacing,\hspacing) [vertex,label=left:$\{1\}$] {};
\node (3) at (\hspacing,\hspacing) [vertex,label=right:$\{3\}$] {};
\node (13) at (0,2*\hspacing) [vertex,label=above:{$\{1,3\}$}] {};
\foreach \x/\y in {0/1,1/13,13/3,3/0}
   {\path (\x) edge[bend left] node [midway,fill=white] {$\emptyset$} (\y);
    \path (\x) edge[bend right] node [midway,fill=white] {$\{2\}$} (\y);}
\end{tikzpicture} &\qquad\qquad &
\begin{tikzpicture}[vertex/.style={circle,draw,fill=blue,minimum size = 2mm,inner sep=0pt},thick]
\setlength{\hspacing}{2cm}
\node (0) at (0,0) [vertex,label=below:{$\emptyset$}] {};
\node (1) at (0,2*\hspacing) [vertex,label=above:{$\{1\}$}] {};
\node (00) at (-\hspacing,\hspacing) {$\emptyset$};
\node (23) at (\hspacing,\hspacing) {$\{2,3\}$};
\path (0) edge[bend left] node [midway,fill=white] {$\{2\}$} (1);
\path (0) edge[bend right] node [midway,fill=white] {$\{3\}$} (1);
\path (0) edge[bend left,in=135] (00);
\path (00) edge[bend left,out=45] (1);
\path (0) edge[bend right,in=225] (23);
\path (23) edge[bend right,out=-45] (1);
\end{tikzpicture}
\end{tabular}
\caption{The graphs $\contractb[3]{R}$ in the cases $R=\{1,3\}$ (left) and $R=\{1\}$ (right). The graphs are formed by contracting the bold edges in the corresponding graph of Figure~\ref{fig:partition}. In each case we get a multigraph with underlying simple graph $Q_R$, and $2^{3-|R|}$ parallel edges for each edge of $Q_R$. The parallel edges may be labelled with the elements of $\power{[3]-R}$.}
\label{fig:quotient}
\end{center}
\end{figure}
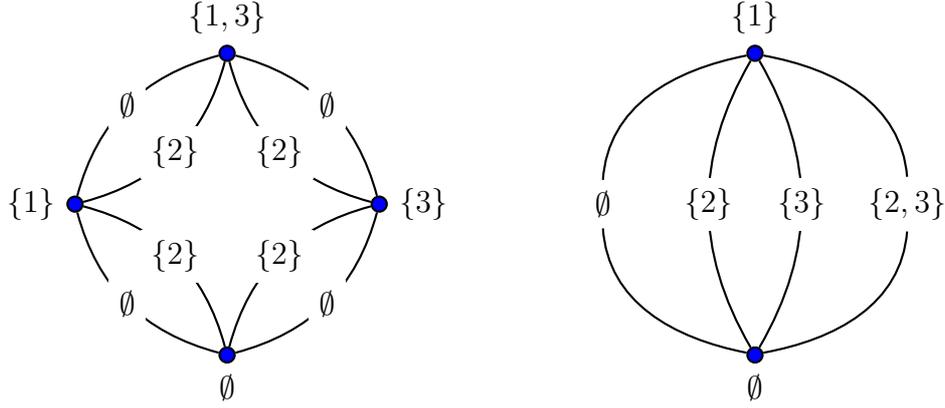

A spanning tree $T$ of \contractb{R}\ corresponds to a choice of spanning tree $T_R=\capr{R}(T)$ of the underlying simple graph $Q_R$, together with a choice of label $Y_e\in\power{\bar{R}}$ for each edge $e$ of $T_R$. We may define edge slides for spanning trees of \contractb{R}\ in an identical manner to edge slides for spanning trees of $Q_n$. For each $i\in [n]$ the automorphism $\sigma_i:Q_n\to Q_n$ descends to a well defined map $\sigma_i:\contractb{R}\to\contractb{R}$, and as before we may define the edge $(e,Y)$ of $T$ to be $i$--slidable if $T-(e,Y)+\sigma_i(e,Y)$ is again a spanning tree of $T$. For $i\in \bar{R}$ this simply corresponds to a change in label from $Y$ to $Y\oplus\{i\}$, so every edge of $T$ is $i$--slidable; while for $i\in R$ this corresponds to a label preserving edge slide in $T_R$, and $(e,Y)$ is $i$--slidable if and only if $e$ is $i$--slidable as an edge of $T_R$. 
We write $\eslidesig{\contractb{R}}$ for the edge slide graph of \contractb{R}, and
for a signature $\mathcal{S}$ of $Q_R$ we write $\eslidesig[\contractb{R}]{\mathcal{S}}$ for the edge slide graph of spanning trees of \contractb{R}\ with signature $\mathcal{S}$. Our discussion above has the following consequence:

\begin{observation}
\label{obs:contraction}
Let $R$ be a proper nonempty subset of $[n]$. For any signature $\mathcal{S}$ of $Q_R$, the edge slide graph $\eslidesig[Q_n/\bar{R}]{\mathcal{S}}$ is connected if and only if $\eslidesig{\mathcal{S}}$ is connected. 
\end{observation}

By a mild abuse of notation we may also regard \capr{R}\ as a map from $Q_n$ to $Q_R$. For each vertex $U$ of $Q_n$ we have 
\[
\capr{R}(U) = U\cap R,
\]
and for each edge $\{U,V\}$ of $Q_n$  we have
\[
\capr{R}(\{U,V\}) 
    = \begin{cases}
      \{\capr{R}(U),\capr{R}(V)\}=\{U\cap R,V\cap R\} 
                              & \text{if $\capr{R}(U)\neq\capr{R}(V)$},\\
      \capr{R}(U)=U\cap R     & \text{if $\capr{R}(U)=\capr{R}(V)$}.
      \end{cases}
\]
If $V=U\oplus\{j\}$ then $\capr{R}(\{U,V\})$ is the edge $\{U\cap R,(U\cap R)\oplus\{j\}\}$ of $Q_R$ if $j\in R$, and is the vertex $U\cap R$ of $Q_R$ if $j\notin R$. Thus $\capr{R}:Q_n\to Q_R$ is not a graph homomorphism in the usual sense, but it is a cellular map if we regard $Q_n$ and $Q_R$ as $1$--dimensional cell-complexes. Note that $Q_n(R,X)$ is the preimage in $Q_n$ of $X\subseteq R$ under $\capr{R}$.

\subsubsection{The Cartesian product of graphs}

Our edge slide graph decompositions will be expressed in terms of the Cartesian product of graphs; see for example~\cite[p.~30]{bondy-murty-2008} or~\cite[D71 (p.~16)]{handbook-graphtheory2014}. The Cartesian product may be defined as follows:

'\begin{figure}
\begin{center}
\setlength{\vspacing}{1.5cm}
\setlength{\hspacing}{1.5cm}
\begin{tikzpicture}[vertex/.style={circle,draw,minimum size = 2mm,inner sep=0pt},thick]
\foreach \y in {0,1,2,3}
  {\foreach \x in {3,5}
     {\node (\x\y) at (\x*\hspacing,\y*\vspacing) [vertex] {};}
      \node (l\y) at (3.66*\hspacing,-0.25*\vspacing+\y*\vspacing) [vertex] {};
      \node (u\y) at (4.33*\hspacing,0.25*\vspacing+\y*\vspacing) [vertex] {};
      \node (1\y) at (1.25\hspacing,\y*\vspacing) [vertex] {};
      \node (2\y) at (2.25\hspacing,\y*\vspacing) [vertex] {};
      \draw [thick] (1\y) -- (2\y);}
\draw [thick] (u1) -- (u2) -- (u3);
\foreach \y in {1,2,3}
      \node (0\y) at (0,\y*\vspacing) [vertex] {};
\foreach \y in {0,1,2,3}
     {\draw [color=white,line width=5] (l\y) -- (5\y);
      \draw [thick] (3\y) -- (u\y) -- (5\y) -- (l\y) -- (3\y);}
\foreach \x in {0,1,2,3,5}
      \draw [thick] (\x1) -- (\x2) -- (\x3);
\draw [color=white,line width=5] (l1) -- (l2) -- (l3);
\draw [thick] (l1) -- (l2) -- (l3);
\node [below right] at (01) {$H$};
\node [below right] at (50) {$G$};
\node [below right] at (51) {$G\boxempty H$};
\end{tikzpicture}
\caption{The Cartesian product of the graphs $G=Q_1\amalg Q_2$ and $H=P_2$, a path of length two.
Notice that $G\boxempty H = (Q_1\amalg Q_2)\boxempty P_2 = (Q_1\boxempty P_2)\amalg (Q_2\boxempty P_2)$. }
\label{fig:cartesian}
\end{center}
\end{figure}
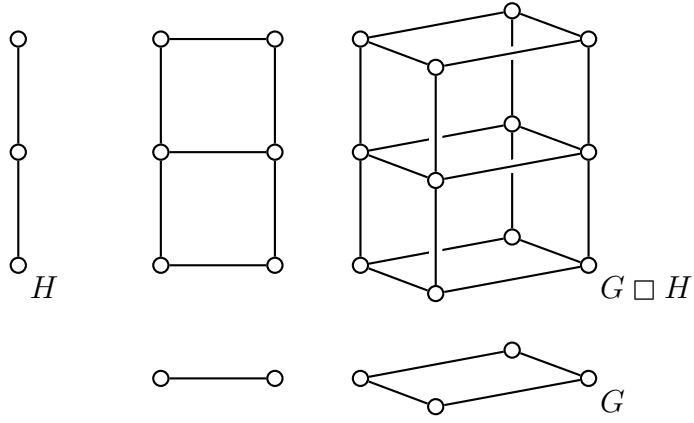

\begin{definition}
Let $G_1=(V_1,E_1)$ and $G_2=(V_2,E_2)$ be graphs, where $V_i$ is the vertex set of $G_i$ and $E_i$ is its edge set. The \textbf{Cartesian product} of $G_1$ and $G_2$, denoted $G_1\boxempty G_2$, is a graph with vertex set $V_1\times V_2$ and edge set $(E_1\times V_2)\cup(V_1\times E_2)$. The incidence relation is as follows:
\begin{itemize}
\item
If $e_1\in E_1$ is incident with $u_1,v_1$, then for all
$w_2\in V_2$, the edge $(e_1,w_2)$ is incident with the vertices $(u_1,w_2)$ and $(v_1,w_2)$ of $G_1\boxempty G_2$. 
\item
Similarly, if  $e_2\in E_2$ is incident with $u_2,v_2\in V_2$, then for all $w_1\in V_1$, the edge $(w_1,e_2)$ is incident with the vertices $(w_1,u_2)$ and $(w_1,v_2)$ of $G_1\boxempty G_2$.
\end{itemize}
Consequently, for simple graphs $G_1$ and $G_2$, the vertices $(u_1,u_2)$ and $(v_1,v_2)$ of $G_1\boxempty G_2$ are adjacent if and only if
\begin{itemize}
\item
$u_2=v_2$, and $u_1$, $v_1$ are adjacent in $G_1$; or
\item
$u_1=v_1$, and $u_2$, $v_2$ are adjacent in $G_2$.
\end{itemize}
We write
\begin{align*}
\bigbox_{i=1}^n G_i &= G_1\boxempty G_2\boxempty\cdots\boxempty G_n, \\
G^{\boxempty n} &= \bigbox_{i=1}^n G = \underbrace{G\boxempty\cdots\boxempty G}_n.
\end{align*}
\end{definition}

Figure~\ref{fig:cartesian} illustrates the Cartesian product of $G=Q_1\amalg Q_2$ and $H=P_2$, a path of length two.

Some sources write $G\times H$ for the Cartesian product of $G$ and $H$, but others use this notation for a different graph product. 
The notation $G\boxempty H$ used here and in~\cite{bondy-murty-2008} avoids this ambiguity and reflects the fact that the product of a pair of edges is a square, as can be seen in Figure~\ref{fig:cartesian}. We note that
\begin{enumerate}
\item
$Q_n \cong (K_2)^{\boxempty n}=(Q_1)^{\boxempty n}$,
and so also $Q_n\boxempty Q_m \cong Q_{n+m}$.
\item
If $G=G_1\amalg G_2$ is a disjoint union of subgraphs $G_1$ and $G_2$, then 
\[
G\boxempty H = (G_1\amalg G_2)\boxempty H = (G_1\boxempty H)\amalg (G_2\boxempty H),
\]
as can be seen in Figure~\ref{fig:cartesian}.
\end{enumerate}

\begin{remark}
If $G_1$ and $G_2$ are regarded as $1$--dimensional cell complexes, then $G_1\boxempty G_2$ is the $1$--skeleton of the $2$--dimensional cell complex $G_1\times G_2$:
\[
G_1\boxempty G_2=(G_1\times G_2)^{(1)}.
\] 
Here $G_1\times G_2$ is the Cartesian product of $G_1$ and $G_2$ as cell complexes; see for example Hatcher~\cite[p.~8]{hatcher-at} . 
\end{remark}

\subsection{Structural characterisation of reducible trees}
\label{sec:structural-characterisation}

Reducible trees may be characterised as follows:

\begin{theorem}
\label{thm:reducible-subtrees}
Let $T$ be a spanning tree of $Q_n$, and let $R$ be a proper nonempty subset of $[n]$.  The following statements are equivalent:
\begin{enumerate}
\item
\label{item:R-is-reducing}
$T$ reduces over $R$. 
\item
\label{item:subtrees}
$T(R,X)$ is a spanning tree of $Q_n(R, X)$ for every $X\subseteq R$.
\item
\label{item:contraction}
$T/\bar{R}$ is a spanning tree of \contractb{R}.
\end{enumerate}
\end{theorem}

Figure~\ref{fig:decomposition} illustrates Theorem~\ref{thm:reducible-subtrees}
for a tree with signature $(1,3,3)$, which reduces over $R=\{1\}$.

\begin{proof}
Let $\sig(T)=\mathcal{S}=(a_1,\dots,a_n)$, and let $E$ be the set of edges of $T$ in directions belonging to $R$. Then $|E|=\sum_{i\in R} a_i$. 
Delete all edges of $T$ belonging to $E$. The resulting graph $T-E=\bigcup_{X\subseteq R} T(R,X)$ has $|E|+1$ components and is a spanning forest\footnote{Recall that we do not require a spanning forest to be a \emph{maximal} spanning forest.} of $G=\bigcup_{X\subseteq R} Q_n(R,X)$, which is the result of deleting all edges of $Q_n$ in directions belonging to $R$. As such, $T(R,X)$ is a spanning tree of $Q_n(R,X)$ for all $X\subseteq R$ if and only if $T-E$ has the same number of components as $G$. But $G$ has $2^{|R|}$ components, so condition~\ref{item:subtrees} holds if and only if
\[
|E|=\sum_{i\in R} a_i = 2^{|R|}-1;
\]
that is, if and only if $\mathcal{S}$ reduces over $R$. 
This proves that condition~\ref{item:R-is-reducing} of the theorem holds if and only if condition~\ref{item:subtrees} does.

We now consider $T/\bar{R}$. This graph is the subgraph of \contractb{R}\ that results from $T$ under the edge contractions transforming $Q_n$ into $\contractb{R}$. Since $T$ is a connected spanning subgraph of $Q_n$, the resultant $T/\bar{R}$ is a connected spanning subgraph of \contractb{R} also. It is therefore a spanning tree if and only if it has $2^{|R|}-1$ edges. But the edges of \contractb{R}\ are exactly the edges of $Q_n$ in directions belonging to $R$, and so the edges of $T/\bar{R}$ are exactly the edges of $T$ in directions belonging to $R$ also. Thus $T/\bar{R}$ has $|E|=\sum_{i\in R}a_i$ edges, and so is a spanning tree if and only if $\sum_{i\in R}a_i=2^{|R|}-1$. This shows that condition~\ref{item:R-is-reducing} holds if and only if condition~\ref{item:contraction} does, completing the proof.
\end{proof}

\begin{corollary}
\label{cor:TcapR}
Let $T$ be a spanning tree of $Q_n$, and let $R$ be a proper nonempty subset of $[n]$.
If $T$ reduces over $R$ then $\capr{R}(T)$ is a spanning tree of $Q_R$. 
\end{corollary}

\begin{proof}
Recall that $\capr{R}:Q_n\to Q_R$ is the map given by the contraction $Q_n\to \contractb{R}$, followed by the projection $\contractb{R}\to Q_R$ to the underlying simple graph.
Thus, the graph $\capr{R}(T)$ is the subgraph of $Q_R$ obtained from $T/\bar{R}$ under the projection $\contractb{R}\to Q_R$. By Theorem~\ref{thm:reducible-subtrees} $T/\bar{R}$ is a spanning tree of \contractb{R}, so $\capr{R}(T)$ is a connected spanning subgraph of $Q_R$. Moreover, since it is acyclic, $T/\bar{R}$ contains at most one edge from each family of parallel edges of \contractb{R}. Therefore the number of edges of $\capr{R}(T)$ is equal to the number of edges $T/\bar{R}$, namely $2^{|R|}-1$. The result follows.
\end{proof}

\begin{figure}
\begin{center}
\setlength{\vspacing}{1.6cm}
\setlength{\hspacing}{1.8cm}
\begin{tabular}{ccc}
\begin{tikzpicture}[vertex/.style={circle,draw,inner sep=1pt,minimum size=2mm},thick]
\node (123) at (0,3*\vspacing) [vertex] {};
\node (12)  at (-\hspacing,2*\vspacing) [vertex]  {};
\node (13)  at (0,2*\vspacing) [vertex]  {};
\node (23)  at (\hspacing,2*\vspacing) [vertex]  {};
\node (1)  at (-\hspacing,\vspacing) [vertex]  {};
\node (2)  at (0,\vspacing) [vertex]  {};
\node (3)  at (\hspacing,\vspacing) [vertex]  {};
\node (0)  at (0,0) [vertex] {};
\foreach \x/\Y in {0/{1,2,3},1/{12,13},2/{12,23},3/{13,23},123/{12,13,23}}
   {\foreach \y in \Y
       \draw (\x) -- (\y);};
\draw [color=white,line width=4] (2)--(12);
\draw  (2)--(12);
\draw [color=white,line width=6pt] (2)--(23);
\draw  (2)--(23);
\node [right] at (0) {$\;\emptyset$};
\node [right] at (3) {$\{3\}$};
\node [right] at (23) {$\{2,3\}$};
\node [right] at (123) {$\{1,2,3\}$};
\node [left] at (12) {$\{1,2\}$};
\node [below right] at (2) {$\{2\}$};
\node [left] at (1) {$\{1\}$};
\node [right] at (13) {$\{1,3\}$};
\draw [line width=3,color=blue] (12) -- (123);
\path (12) edge[line width=3,color=blue] node[midway,above left] {$T(\{1\},\{1\})$} (123);
\draw [line width=3,color=blue] (1) -- (13);
\path (0) edge[line width=3,color=blue] node[midway,below right] {$T(\{1\},\emptyset)$} (3);
\draw [line width=3,color=blue] (3) -- (23);
\draw [line width=3,color=blue] (13) -- (123);
\draw [color=white,line width=5] (2)--(12);
\draw [line width=3,color=red,dashed] (12) -- (2);
\draw [line width=3,color=blue] (2) -- (0);
\end{tikzpicture}& \hspace{\hspacing} &
\begin{tikzpicture}[vertex/.style={circle,draw,color=black,fill=blue,minimum size = 2mm,inner sep=0pt},thick]
\setlength{\hspacing}{2cm}
\node (0) at (0,0) [vertex,label=below:{$\emptyset$}] {};
\node (1) at (0,2*\hspacing) [vertex,label=above:{$\{1\}$}] {};
\node (00) at (-\hspacing,\hspacing) {$\emptyset$};
\node (23) at (\hspacing,\hspacing) {$\{2,3\}$};
\path (0) edge[line width=3,bend left,color=red,dashed] node [midway,fill=white] {$\{2\}$} (1);
\path (0) edge[bend right] node [midway,fill=white] {$\{3\}$} (1);
\path (0) edge[bend left,in=135] (00);
\path (00) edge[bend left,out=45] (1);
\path (0) edge[bend right,in=225] (23);
\path (23) edge[bend right,out=-45] (1);
\end{tikzpicture}
\end{tabular}
\caption{\emph{Left:} A spanning tree $T$ of $Q_3$ with signature $(1,3,3)$, which reduces over $R=\{1\}$. The solid bold edges show the spanning trees $T(R,X)$ of $Q_3(R,X)$ for $X\subseteq R$, and the bold dashed edge is the (here, unique) edge of $T$ in a direction belonging to $R$. \emph{Right:} The result of contracting the subcubes $Q_3(R,X)$ for $X\subseteq R$. The bold dashed edge $\bigl\{\{2\},\{1,2\}\bigr\}$ of $T$ becomes the bold dashed spanning tree $T/\bar{R}$ of \contractb[3]{R}.
The tree $T$ can be completely reconstructed from the spanning trees $T(R,X)$, together with the spanning tree $T/\bar{R}$.}
\label{fig:decomposition}
\end{center}
\end{figure}
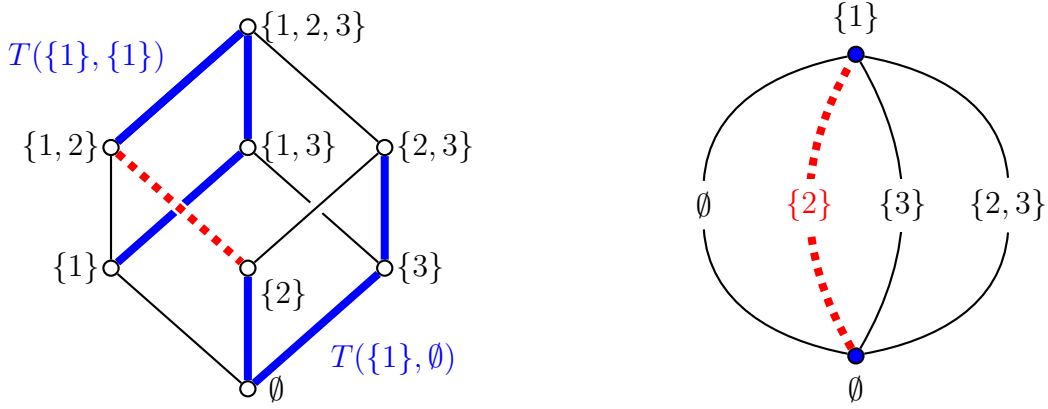

\subsection{Decomposing reducible trees}
\label{sec:decompose}

In view of Theorem~\ref{thm:reducible-subtrees}, we may canonically define a map
\[
\Psi_R : \rtree{R}\to\tree{\contractb{R}}\times\prod_{X\subseteq R}\tree{Q_n(R,X)}
\]
by setting 
\[
\Psi_R(T) = \bigl(T/\bar{R},\bigl(T(R,X)\bigr)_{X\subseteq R}\bigr).
\]
We show below in Theorem~\ref{thm:rtrees-bijection} that this map is a bijection, and then in Theorem~\ref{thm:redR} that it in fact defines an isomorphism of edge slide graphs. 

\begin{theorem}
\label{thm:rtrees-bijection}
Let $R$ be a proper non-empty subset of $[n]$. The map
\[
\Psi_R : \rtree{R}\to\tree{\contractb{R}}\times\prod_{X\subseteq R}\tree{Q_n(R,X)}
\]
defined by 
\[
\Psi_R(T) = \bigl(T/\bar{R},\bigl(T(R,X)\bigr)_{X\subseteq R}\bigr)
\]
is a bijection. 
\end{theorem}

\begin{proof}
The edges of $Q_n$ may be naturally identified with the edges of $(\contractb{R})\cup\bigcup_{X\subseteq R}Q_n(R,X)$. Using this identification we see that if $\Psi_R(T_1)=\Psi_R(T_2)$ then the edge set of $T_1$ is equal to the edge set of $T_2$, so $T_1=T_2$. Therefore $\Psi_R$ is one-to-one. 

It remains to show that $\Psi_R$ is onto. Let 
\[
\mathcal{T}=\bigl(T_R,\bigl(T^X\bigr)_{X\subseteq R}\bigr)
\in\tree{\contractb{R}}\times\prod_{X\subseteq R}\tree{Q_n(R,X)}. 
\]
The edges of \contractb{R}\ may be canonically identified with the edges of $Q_n$ in directions belonging to $R$, and using this identification we define $T$ be the subgraph of $Q_n$ with edge set
\[
E(T) = E(T_R)\cup\bigcup_{X\subseteq R} E(T^X).
\]
We claim that $T$ is a spanning tree of $Q_n$ that reduces over $R$, and that $\Psi_R(T)=\mathcal{T}$. 

To see this, first note that the subcubes $Q_n(R,X)$ partition the edges of $Q_n$ in directions belonging to $\bar{R}$. Thus
\begin{align*}
|E(T)| &= |E(T_R)| + \sum_{X\subseteq R} |E(T^X)| \\
       &=(2^{|R|}-1)+2^{|R|}(2^{n-|R|}-1)  \\
       &=2^n-1.
\end{align*} 
Next, recall that \contractb{R}\ is obtained from $Q_n$ by contracting each subcube $Q_n(R,X)$ to a single vertex. Since $T^X$ is a spanning tree of $Q_n(R,X)$ for each $X$, and $T_R$ is a spanning tree of \contractb{R}, it follows that $T$ is a spanning subgraph of $Q_n$. Since it has $2^n-1$ edges it is therefore a spanning tree. 
Moreover $T$ has $|E(T_R)|=2^{|R|}-1$ edges in directions belonging to $R$, so $T$ reduces over $R$. Thus $T\in\rtree{R}$, and it's clear by construction that we have $\Psi_R(T)=\mathcal{T}$. 
\end{proof}

\begin{observation}
\label{obs:productsignature}
Let $\mathcal{S}$ be a signature belonging to $\red{R}$, and let $T\in\rtree{R}$ be such that $\Psi_R(T)=\mathcal{T}=(T_R,(T^X)_{X\subseteq R})$. Then 
\[
\sig(T)=\mathcal{S} \qquad\Leftrightarrow\qquad
\sig(T_R)=\mathcal{S}|_R\text{ and }\sum_{X\subseteq R}\sig(T^X)=\mathcal{S}|_{[n]-R}.
\] 
\end{observation}

\begin{example} 
For the tree $T$ of Figure~\ref{fig:decomposition}, with  $\mathcal{S}=\sig(T)=(1,3,3)$ we have $\sig(T_{\{1\}})=(1)=\mathcal{S}|_{\{1\}}$, and
\[
\sig(T(\{1\},\emptyset))+\sig(T(\{1\},\{1\}))=(2,1)+(1,2)=(3,3)=\mathcal{S}|_{\{2,3\}}. 
\]

\end{example}

\subsection{The edge slide graph isomorphism theorem for reducible trees}
\label{sec:reducibleslides}

The bijection $\Psi_R$ of Theorem~\ref{thm:rtrees-bijection} has domain \rtree{R}, which is the vertex set of the edge slide graph \eslidesig{\red{R}}. In this section we show that $\Psi_R$ in fact defines a graph isomorphism between \eslidesig{\red{R}}\ and a suitable Cartesian product of smaller edge slide graphs:

\begin{theorem}
\label{thm:redR}
Let $R$ be a proper nonempty subset of $[n]$, and let $r=|R|$. Then 
\begin{align*}
\eslidesig{\red{R}}=\bigcup_{\mathcal{S}\in\red{R}} \eslidesig{\mathcal{S}} 
   &\cong \eslidesig{\contractb{R}} \boxempty \bigbox_{X\subseteq R} \eslidesig{Q_n(R,X)} \\
   & \cong \eslidesig{\contractb{R}} \boxempty (\eslide{n-r})^{\boxempty 2^{r}}. \\
\end{align*}
\end{theorem}

Theorem~\ref{thm:redR} follows from Theorem~\ref{thm:rtrees-bijection}
and the following characterisation of edge slides in a reducible tree:

\begin{theorem}
\label{thm:reducibleslides}
Let $R$ be a proper nonempty subset of $[n]$, and let $T$ be a spanning tree of $Q_n$ that reduces over $R$. Let $e=\{Y,Y\oplus\{j\}\}$ be an edge of $T$ in direction $j$, and set $X=Y\cap R$. 
\begin{enumerate}
\item
If $j\notin R$, then $e\in Q_n(R,X)$ and
\begin{enumerate}
\item\label{item:invariance}
$e$ is not slidable in any direction $i\in R$;
\item
$e$ is slidable in direction $i\notin R$ if and only if it is $i$--slidable as an edge of $T(R,X)$. 
\end{enumerate}
\item
If $j\in R$, then $e$ is $i$--slidable if and only if it is $i$--slidable as an edge of $T/\bar{R}$. Consequently
\begin{enumerate}
\item
$e$ is slidable in any direction $i\notin R$;
\item
$e$ is slidable in direction $i\in R$ if and only if $\capr{R}(e)$ is $i$--slidable as an edge of the tree $\capr{R}(T)$. 
\end{enumerate}
\end{enumerate}
\end{theorem}

\begin{proof}
Let $i\in [n]$ be such that $i\neq j$. By Lemma~\ref{lem:slidable}, for $e$ to be $i$--slidable we require that $\sigma_i(e)$ does not belong to $T$, and that the resulting cycle $C$ in $T+ \sigma_i(e)$ created by adding $\sigma_i(e)$ to $T$ also contains $e$. 
We consider four possibilities, according to whether $i$ and $j$ belong to $R$. 

\begin{enumerate}
\item
Suppose first that $j\notin R$. Then $\sigma_i(e)$ lies in $Q_n(R,X')$, where $X'=(X\oplus\{i\})\cap R$. Since $T(R,X')$ is a spanning tree of $Q_n(R,X')$, if $\sigma_i(e)$ does not belong to $T$ then the cycle $C$ lies entirely in $Q_n(R,X')$. 
\begin{enumerate}
\item 
If $i\in R$ then $X'=X\oplus\{i\}\neq X$, so $e$ does not belong to $C$. It follows that $e$ is not $i$--slidable. 
\item
If $i\notin R$ then $X'=X$, and $\sigma_i(e)$ does not belong to $T$ if and only if it does not belong to $T(R,X)$. If that is the case then $C$ is the cycle in $T(R,X)+\sigma_i(e)$ created by adding $\sigma_i(e)$ to $T(R,X)$, and it follows that $e$ is $i$--slidable as an edge of $T$ if and only if it is $i$--slidable as an edge of $T(R,X)$. 
\end{enumerate}
\item
Suppose now that $j\in R$. The edges of $\contractb{R}$ in direction $j$ may be naturally identified with the edges of $Q_n$ in direction $j$, and with respect to this identification, for any $i\neq j$ the edge $\sigma_i(e)$ belongs to $T$ if and only if it belongs to $T/\bar{R}$. It follows that if $\sigma_i(e)$ belongs to $T$ then $e$ is $i$--slidable in neither $T$ nor $T/\bar{R}$. So suppose that $\sigma_i(e)$ does not belong to $T$, and let $P$ be the path in $T$ from one endpoint of $\sigma_i(e)$ to the other. Write  $P=v_0v_1\dots v_\ell$, and for $0\leq a\leq\ell$ let $X_a\subseteq R$ be such that $v_a\in Q_n(R,X_a)$; that is, $X_a=v_a\cap R$. Note that $C=P+\sigma_i(e)$.

For any $0\leq a\leq b\leq\ell$, the subpath $v_av_{a+1}\cdots v_b$ is the unique path in $T$ from $v_a$ to $v_b$. If $X_a=X_b$ then $v_a$ and $v_b$ both belong to $Q_n(R,X_a)$, so this path must be the unique path from $v_a$ to $v_b$  inside the spanning tree $T(R,X_a)$ of $Q_n(R,X_a)$. Therefore $X_c=X_a$ for $a\leq c\leq b$. It follows that for all $X'\subseteq R$, if $P(R,X')=P\cap Q(R,X')$ is nonempty then it consists of a single path. 

Consequently, when the subcubes 
$Q_n(R,X')$ are contracted to form $\contractb{R}$, the resulting subgraph  $C/\bar{R}$ is still a cycle, because it is a contraction of $C$ in its own right. This cycle is the cycle in $T/\bar{R}+\sigma_i(e)$ that is created when $\sigma_i(e)$ is added to $T/\bar{R}$, and so it contains both $e$ and $\sigma_i(e)$ if and only if both edges belong to $C$. It follows that $e$ is $i$--slidable in $T$ if and only if it is $i$--slidable in $T/\bar{R}$. 

As an edge of $T/\bar{R}$ the endpoints of $e$ are $X$ and $X\oplus\{j\}$, and the endpoints of $\sigma_i(e)$ are $X'=(X\oplus\{i\})\cap R$ and $X'\oplus\{j\}$. We now consider two cases according to whether or not $i\in R$. 
\begin{enumerate}
\item 
If $i\notin R$ then $X=X'$ and the end points of $e$ and $\sigma_i(e)$ in $T/\bar{R}$ co-incide. Therefore $C/\bar{R}$ must consist of $e$ and $\sigma_i(e)$ only, and so contains both edges. Therefore $e$ is $i$--slidable.
\item
If $i\in R$ then the endpoints of $e$ and $\sigma_i(e)$ in $T/\bar{R}$ differ. If an edge $f$ parallel to $\sigma_i(e)$ belongs to $T/\bar{R}$ then $C/\bar{R}$ consists of $f$ and $\sigma_i(e)$ only, and $e$ is not $i$--slidable in $T/\bar{R}$. In this case $\capr{R}(f)=\sigma_i(\capr{R}(e))$ belongs to $\capr{R}(T)$, so $\capr{R}(e)$ is not $i$--slidable in $\capr{R}(T)$ either.

Otherwise, $\capr{R}(\sigma_i(e))=\sigma_i(\capr{R}(e))$ does not belong to $\capr{R}(T)$, and the cycle $C'$ created by adding $\sigma_i(\capr{R}(e))$ to $\capr{R}(T)$ is $\capr{R}(C/\bar{R})=\capr{R}(C)$. Since $T/\bar{R}+\sigma_i(e)$ contains at most one edge from each parallel family of edges in \contractb{R}, the cycle $C/\bar{R}$ contains both $e$ and $\sigma_i(e)$ if and only if $C'$ contains both $\capr{R}(e)$ and $\sigma_i(\capr{R}(e))$. It follows that $e$ is $i$--slidable in $T/\bar{R}$ if and only if $\capr{R}(e)$ is $i$--slidable in $\capr{R}(T)$, as claimed. 
\end{enumerate}
\end{enumerate}
\end{proof}

\begin{proof}[Proof of Theorem~\ref{thm:redR}]
The vertex set of $\eslidesig{\red{R}}$ is $\rtree{R}$, and 
the vertex set of the product $\eslidesig{\contractb{R}} \boxempty \bigbox_{X\subseteq R} \eslidesig{Q_n(R,X)}$ is $\tree{\contractb{R}}\times\prod_{X\subseteq R}\tree{Q_n(R,X)}$. By Theorem~\ref{thm:rtrees-bijection} the function $\Psi_R$ is a bijection between the vertex sets of 
$\eslidesig{\red{R}}$ and $\eslidesig{\contractb{R}} \boxempty \bigbox_{X\subseteq R} \eslidesig{Q_n(R,X)}$. 

Let $T\in \rtree{R}$, and let $e$ be an edge of $T$. Then Theorem~\ref{thm:reducibleslides} shows that 
$e$ is slidable as an edge of $T$ if and only if it is slidable as an edge of whichever tree $T/\bar{R}$ or $T(R,X)$ it belongs to with respect to the decomposition $\Psi_R(T)$ of Theorem~\ref{thm:rtrees-bijection}. Moreover, the proofs of these theorems show that when $e$ is slidable, the edge slide and the decomposition commute: if $T'$ is the result of sliding $e$ in $T$, then $T'/\bar{R}$ or $T'(R,X)$ (as applicable) is the result of sliding $e$ in the decomposition. It follows that $\Psi_R$ is a graph homomorphism, completing the proof.
\end{proof}
 
As a corollary to part~\ref{item:invariance} of Theorem~\ref{thm:reducibleslides} we obtain the following:
\begin{corollary}
\label{cor:invariantsubsignatures}
Let $R$ be a proper nonempty subset of $[n]$, and let $T$ be a spanning tree of $Q_n$ that reduces over $R$.
For all $X\subseteq R$, the signature of $T(R,X)$ is an invariant of the connected component of $\eslide{n}$ containing $T$. More precisely, suppose that $T'$ can be obtained from $T$ by edge slides. Then $\sig(T'(R,X))=\sig(T(R,X))$ for all $X\subseteq R$. 
\end{corollary}

\begin{example}
Refer again to the tree $T$ in Figure~\ref{fig:decomposition}, which reduces over $\{1\}$. No edge of $T(\{1\},\emptyset)$ or $T(\{1\},\{1\})$ may be slid in direction 1. The only slidable edge of $T(\{1\},\emptyset)$ is $\{\emptyset,\{3\}\}$, which may be slid in direction 2 only; and the only slidable edge of $T(\{1\},\{1\})$ is $\{\{1,3\},\{1,2,3\}\}$, which may be slid in direction 3 only. The edge $\{\{2\},\{1,2\}\}$ in direction $1\in R$ may be freely slid in either direction 2 or 3. These edge slides all leave $\sig(T(\{1\},\emptyset))=(2,1)$ and $\sig(T(\{1\},\{1\}))=(1,2)$ unchanged. 
\end{example}

\subsection{Special cases}
\label{sec:special}

We now apply Theorem~\ref{thm:redR} in several special cases. We first consider the case $R=\{1\}$, and then use this to show that the edge slide graph of a supersaturated signature has the isomorphism type of a cube. We then apply this in turn to express the edge slide graph of a saturated signature in terms of an edge slide graph associated with its unsaturated part.

\begin{theorem}
\label{thm:red1n}
Let $n\geq 2$. Then
\[
\eslidesig{\red{\{1\}}} \cong Q_{n-1}\boxempty(\eslide{n-1})^{\boxempty 2}
    \cong \coprod_{\mathcal{S}_1,\mathcal{S}_2\in\Sig(Q_{n-1})}
       Q_{n-1}\boxempty\eslidesig{\mathcal{S}_1}\boxempty\eslidesig{\mathcal{S}_2}.
\]
For $\mathcal{S}=(1,\mathcal{S}')\in\red{\{1\}}$ we have
\[
\eslidesig{\mathcal{S}} 
    \cong \coprod_{\mathcal{S}_1+\mathcal{S}_2=\mathcal{S}'}
       Q_{n-1}\boxempty\eslidesig{\mathcal{S}_1}\boxempty\eslidesig{\mathcal{S}_2}.
\]
\end{theorem}

\begin{proof}
For compactness of notation let $R=\{1\}$. By Theorem~\ref{thm:redR} we have
\[
\eslidesig{\red{R}} \cong \eslidesig{\contractb{R}}\boxempty(\eslide{n-1})^{\boxempty 2}.
\]
The graph $\contractb{R}$ is a multigraph with underlying simple graph $Q_1$, and $2^{n-1}$ parallel edges labelled with the subsets of $\bar{R}=\{2,\ldots,n\}$. A spanning tree of \contractb{R}\ consists of a single edge, which may be canonically identified with its label in $\power{\bar{R}}$. Two such trees are related by an edge slide in direction $j$ precisely when their labels differ by adding or deleting $j$, so 
\[
\eslidesig{\contractb{R}} \cong Q_{[n]-{R}} \cong Q_{n-1}.
\]
This gives the first isomorphism. The second then follows from the fact that
\[
\eslide{n-1} = \coprod_{\mathcal{S}\in\Sig(Q_{n-1})} \eslidesig{\mathcal{S}}.
\]

For the final assertion, under the isomorphisms a vertex $(Y,T_1,T_2)$ of $Q_{n-1}\boxempty\eslidesig{\mathcal{S}_1}\boxempty\eslidesig{\mathcal{S}_2}$ corresponds to a tree $T\in\eslidesig{\red{R}}$ such that
\begin{align*}
T(R,\emptyset) &= T_1, & 
T(R,R) &= T_2. 
\end{align*}
The signature of $T$ is given by
\[
\sig(T) = (1,\sig(T_1)+\sig(T_2)) = (1,\mathcal{S}_1+\mathcal{S}_2),
\]
from which the claim follows.
\end{proof}

\begin{example}
\label{ex:red31}
We apply Theorem~\ref{thm:red1n} to determine the components of $\red[3]{\{1\}}$. We have
\[
\eslide{2} = \eslidesig{1,2}\amalg\eslidesig{2,1}\cong Q_1\amalg Q_1,
\]
so 
\[
\red[3]{\{1\}} 
    \cong Q_2\boxempty[\eslidesig{1,2}\amalg\eslidesig{2,1}]^{\boxempty 2}
    \cong Q_2\boxempty[Q_1\amalg Q_1]^{\boxempty 2}.
\]
Therefore
$\red[3]{\{1\}}$ has four components, each isomorphic to $Q_2\boxempty Q_1\boxempty Q_1\cong Q_4$. Trees belonging to the component $Q_2\boxempty\eslidesig{a,b}\boxempty\eslidesig{c,d}$ have signature $(1,a+c,b+d)$, so
\begin{align*}
\eslidesig{1,2,4} &\cong Q_2\boxempty(\eslidesig{1,2})^{\boxempty 2} 
                   \cong Q_4, \\
\eslidesig{1,3,3} &\cong \bigl(Q_2\boxempty\eslidesig{1,2}\boxempty\eslidesig{2,1}\bigr)\amalg
       \bigl(Q_2\boxempty\eslidesig{2,1}\boxempty\eslidesig{1,2}\bigr) 
                  \cong Q_4\amalg Q_4, \\
\eslidesig{1,4,2} &\cong Q_2\boxempty(\eslidesig{2,1})^{\boxempty 2}
                  \cong Q_4.
\end{align*}
Up to permutation there is just one remaining signature of $Q_3$, namely the irreducible signature $(2,2,3)$. This signature is connected, and the structure of $\eslidesig{2,2,3}$ has been determined by Henden~\cite{henden-2011}.
\end{example}

As a corollary to Theorem~\ref{thm:red1n} we show that the edge slide graph of a supersaturated signature has the isomorphism type of a cube:
\begin{corollary}
\label{cor:Esupersat}
For the supersaturated signature $\supersat=(1,2,4,\dots,2^{n-1})$ we have
\[
\eslidesig{\supersat} \cong Q_{2^n-n-1}.
\]
\end{corollary}

\begin{proof}
The proof is by induction on $n$, with the technique used to find \eslidesig{1,2,4} in Example~\ref{ex:red31} providing the inductive step. We have previously found
\begin{align*}
\eslidesig{1} &\cong Q_0, & 
\eslidesig{1,2} &\cong Q_1, & 
\eslidesig{1,2,4} &\cong Q_4
\end{align*}
so the result is already established for $n\leq 3$. We may therefore use any one of these cases as the base for the induction.

For the inductive step, suppose that the result holds for \supersat[n-1]. The signature \supersat\ belongs to \red{\{1\}}, so by Theorem~\ref{thm:red1n} it is a disjoint union of subgraphs of \eslide{n}\ of the form $Q_{n-1}\boxempty\eslidesig{\mathcal{S}_1}\boxempty\eslidesig{\mathcal{S}_2}$. Such a subgraph lies in \eslidesig{\supersat}\ precisely when 
\[
\mathcal{S}_1+\mathcal{S}_2=(2,4,\dots,2^{n-2}),
\]
and it follows easily from the characterisation of signatures Theorem~\ref{thm:characterisation} that the only possibility is $\mathcal{S}_1=\mathcal{S}_2=\supersat[n-1]$. We therefore have
\begin{align*}
\eslidesig{\supersat} 
  &\cong 
   Q_{n-1}\boxempty \eslidesig{\supersat[n-1]}\boxempty\eslidesig{\supersat[n-1]}\\
  &\cong Q_{n-1}\boxempty (Q_{2^{n-1}-n})^{\boxempty 2}  \\
  &\cong Q_{2^n-n-1}.
\end{align*}
This establishes the inductive step.
\end{proof}

As our final special case, we use Corollary~\ref{cor:Esupersat} to show that the edge slide graph of a saturated signature may be expressed in terms of an edge slide graph associated with its unsaturated part:
\begin{corollary}
\label{cor:Esat}
Suppose that the ordered signature $\mathcal{S}=(a_1,\dots,a_n)$ is saturated above direction $r$. Let $R=[r]$ and let $S'=(a_1,\dots,a_r)$. Then
\begin{align*}
\eslidesig{\mathcal{S}} 
  &\cong 
   \eslidesig[Q_n/\bar{R}]{\mathcal{S}'}\boxempty(Q_{2^{n-r}-(n-r)-1})^{\boxempty 2^r}\\
  &\cong\eslidesig[Q_n/\bar{R}]{\mathcal{S}'}\boxempty Q_{N},
\end{align*}
where $N=2^r(2^{n-r}-(n-r)-1)$. 
\end{corollary}

\begin{proof}
The signature $\mathcal{S}$ reduces over $R$, so $\eslidesig{\mathcal{S}}$ consists of one or more connected components of $\eslidesig{\red{R}}$. By Theorem~\ref{thm:redR} we have
\[
\eslidesig{\red{R}}\cong \eslidesig{\contractb{R}} \boxempty \bigbox_{X\subseteq R} \eslidesig{Q_n(R,X)},
\]
and by Observation~\ref{obs:productsignature} a vertex $(T_R,(T^X)_{X\subseteq R})$ of this product corresponds to a tree with signature $\mathcal{S}$ if and only if
\[
\sig(T_R)=\mathcal{S}|_R=\mathcal{S}'\text{ and }\sum_{X\subseteq R}\sig(T^X)=\mathcal{S}|_{\bar{R}} = (2^{r},2^{r+1},\dots,2^{n-1}).
\]
Let $\sig(T^X) = (e^X_{r+1},\dots,e^X_n)$ for each $X\subseteq R$. Then an easy induction on $j$ using the signature condition shows that $e^X_{r+j}= 2^j$ for all $X\subseteq R$, so that $\sig(T^X)=\supersat[n-r]$ for all $X$. Then
\[
\eslidesig{\mathcal{S}} \cong
   \eslidesig[Q_n/\bar{R}]{\mathcal{S}'}\boxempty\bigbox_{X\subseteq R}\eslidesig{\supersat[n-r]}, 
\]
and the result now follows by Corollary~\ref{cor:Esupersat}. 
\end{proof}

\section{Strictly reducible signatures are disconnected}
\label{sec:disconnected}

Our last major result is that the edge slide graph of a strictly reducible signature is disconnected:

\begin{theorem}
\label{thm:disconnected}
Let $\mathcal{S}=(a_1, \dots, a_n)$ be a strictly reducible signature of $Q_n$.  Then the edge slide graph $\mathcal{E}(\mathcal{S})$ is disconnected.
\end{theorem}

The reason underlying Theorem~\ref{thm:disconnected} is illustrated by the edge slide graph of the strictly reducible signature $(1,3,3)$, which we saw in Example~\ref{ex:red31} breaks into two components: one consisting of the trees $T$ such that $\sig(T(\{1\},\emptyset))=(1,2)$, and a second consisting of the trees $T$ such that $\sig(T(\{1\},\emptyset))=(2,1)$. 
Recall that by Corollary~\ref{cor:invariantsubsignatures}, if $T$ reduces over $R$, then for all $X\subseteq R$ the signature $\sig(T(R,X))$ is an invariant of the connected component of \eslide{n}\ containing $T$. Thus, we can show that \eslidesig{\mathcal{S}}\ is disconnected by showing there exist $X\subseteq R$ and trees $T$ and $T'$ with signature $\mathcal{S}$ such that $\sig(T(R,X))\neq\sig(T'(R,X))$. 

To find the required trees $T$, $T'$ it suffices to prove the existence of \emph{single} tree $T$ for which there are subsets $X,Y\subseteq R$ such that $\sig(T(R,X))\neq\sig(T(R,Y))$: given such a tree, we may obtain $T'$ by simply exchanging $T(R,X)$ and $T(R,Y)$. In Lemma~\ref{lem:differentsubsignatures} we prove the existence of such a tree, for a suitable choice of reducing set $R$.

\begin{lemma}
\label{lem:differentsubsignatures}
Let $\mathcal{S}= (a_1, \dots, a_n)$ be an ordered strictly reducible signature with unsaturated part $\mathcal{S}'=(a_1, \dots, a_s)$, and let $r<s-1$ be such that $\mathcal{S}'$ reduces over $[r]$ but not $[r+1]$. Let $R=[r]$.
 Then for any distinct $X, Y\subseteq  R$, there exists a spanning tree $T$ of $Q_n$ with signature $\mathcal{S}$ such that $T(R,X)$ and $T(R,Y)$ have different signatures. 
\end{lemma} 

\begin{remark}
\label{rem:choose.r}
Note that $r$ as required above necessarily exists.  Since $\mathcal{S}'$ is reducible it reduces over $[t]$ for some $t<s$, and since it is unsaturated it does not reduce over $[s-1]$. Thus we may for instance take $r$ to be
the largest integer $t<s$ such that $\mathcal{S}'$ reduces over $[t]$.
\end{remark}

\begin{proof}[Proof of Lemma~\ref{lem:differentsubsignatures}]
Let $T$ be a spanning tree of $Q_n$ with signature $\mathcal{S}$. If there exists $Z\subseteq R$ such that $T(R,X)$ and $T(R,Z)$ have different signatures then we can construct the required tree by (if necessary) swapping $T(R,Y)$ and $T(R,Z)$. So suppose that this is not the case. Then the subtrees $T(R,Z)$ have the same signature for all $Z\subseteq R$. Let $\mathcal{U}=(e_{r+1}, \dots, e_n)$ be this common signature. 

For any $i\in\{r+1,\ldots,n\}$ each edge of $T$ in direction $i$ lies in $T(R,Z)$ for some $Z\subseteq R$, and since each such tree contains $e_i$ edges in direction $i$ we have $a_i = 2^re_i$. It follows that $\mathcal{U}$ is ordered. 
We begin by showing under our choice of $r$ that $e_{r+1}\geq 2$.
Suppose to the contrary that $e_{r+1}=1$. Then $a_{r+1}=2^r$, and consequently
\[
\sum_{i=1}^{r+1} a_i = \sum_{i=1}^r a_i + a_{r+1} = 2^r-1+2^r=2^{r+1}-1,
\]
so $\mathcal{S}'$ reduces over $[r+1]$. This contradicts the choice of $r$, so we must have $e_{r+1}\geq 2$ as claimed, which then  forces $e_{r+2}\geq 2$ also because $\mathcal{U}$ is ordered.

Consider 
\begin{align*}
\mathcal{U}_1 &=(e_{r+1}-1, e_{r+2}+1, e_{r+3}, \dots, e_n),\\
\mathcal{U}_2 &=(e_{r+1}+1, e_{r+2}-1, e_{r+3}, \dots, e_n),
\end{align*}
and note that $\mathcal{U}_1+ \mathcal{U}_2=2\mathcal{U}$. 
We show that $\mathcal{U}_1$ and $\mathcal{U}_2$ are signatures of $Q_{n-r}$, so  there exist spanning trees $T_1$ and $T_2$ of $Q_{n-r}$ with signatures $\mathcal{U}_1$ and $\mathcal{U}_2$ respectively. For simplicity, we let $f_i=e_{r+i}$ for $i=1, 2, \dots, n-r$. We consider  $\mathcal{U}_1$ and $\mathcal{U}_2$ separately.

For $\mathcal{U}_1$, we distinguish the following cases according to whether or not $f_2< f_3$.
\begin{enumerate} 
\item Suppose $f_1\leq f_2< f_3\leq \dots \leq f_{n-r}$. Then we have $f_1-1< f_{2}+1\leq f_3$. Write
\[
\mathcal{U}_1=(f_1-1, f_2+1, f_3, \dots, f_{n-r})= (f_{1}', f_{2}', f_{3}', \dots, f_{n-r}').
\]
Then
\[
f_{1}'< f_{2}'\leq f_{3}'\leq \dots \leq f'_{n-r}
\]
is in nondecreasing order; and  
\[
\sum_{i=1}^k f'_i=
\begin{cases}
 f_1-1\geq 1,  &\text{for $k=1$},\\
\sum_{i=1}^k f_i\geq 2^k-1,                 &\text{for $2\leq k\leq n-r$},
\end{cases}
\]
with equality in the second case when $k=n-r$.   
We conclude that $\mathcal{U}_1$ is a signature of $Q_{n-r}$, by Theorem~\ref{thm:characterisation}.

\item
Suppose $f_{1}\leq f_{2}= f_{3}= \dots = f_{p}< f_{p+1}\leq \dots\leq f_{n-r}$ for some $p$, with $3\leq p\leq n-r$. Then we have 
\[
f_{1}-1< f_{3}=\dots =f_{p}< f_{2}+1\leq f_{p+1}\leq \dots \leq f_{n-r}.
\]
Let  
\[
\mathcal{U}_1'= (f_{1}', f_2',\dots, f_{p}', f_{p+1}', \dots, f_{n-r}')=(f_{1}-1, f_{3},\dots, f_p, f_{2}+1, f_{p+1}, \dots, f_{n-r}).
\]
Then $\mathcal{U}_1'$ is an ordered permutation of $\mathcal{U}_1$, so it suffices to show that $\mathcal{U}_1'$ is a signature. The only sums $\sum_{i=1}^k f_i'$ that are not equal to the corresponding sum $\sum_{i=1}^k f_i$ are the sums
\[
\sum_{i=1}^j f_i'= f_{1}-1+ (j-1) f_2,
\]
 
for $1\leq j< p$. We therefore consider the value of $f_{1}-1+ (j-1) f_2$ for $1\leq j< p$. For all $1\leq y\leq p$ let  
\[
f(y)=\sum_{i=1}^y f_i=f_{1}+(y-1)f_{2},
\]
and let 
\[
g(y)=2^y-1.
\]
Then 
\[
f(1)= f_{1}\geq 2 > 1=g(1),
\]
and 
\[
f(p)=\sum_{i=1}^p f_i\geq 2^p-1=g(p).
\]
To verify that $\mathcal{U}_1$ is a signature of $Q_{n-r}$, it remains to show that $g(j)< f(j)$, for all $1<j<p$. Since $g$ is convex, for any $0\leq t\leq 1$, we have
\[
g((1-t)+tp)\leq (1-t)g(1)+t g(p).
\]

Let $t=\frac{j-1}{p-1}$.  Then for $1\leq j\leq p$ we have $0\leq t\leq 1$ and $1-t+tp = j$. So 
\begin{align*}
g(j)  &\leq \left(1-\frac{j-1}{p-1}\right) g(1)+ \frac{j-1}{p-1} g(p)\\
&< \left(1-\frac{j-1}{p-1}\right) f(1) +\frac{j-1}{p-1} f(p) \\
&=\left(1-\frac{j-1}{p-1}\right) f_{1}+\frac{j-1}{p-1}(f_{1}+(p-1)f_{2})\\
&=f_{1}+(j-1)f_{2}=f(j).
\end{align*}
Therefore
\[
\sum_{i=1}^j f_i'= \sum_{i=1}^j f_i-1=f(j)-1\geq 2^j-1,
\]
showing that $\mathcal{U}_1$ satisfies the signature condition. 
\end{enumerate} 

For $\mathcal{U}_2$, we consider the following cases according to whether $f_2=f_1,$ $f_2=f_1+1$ or $f_2> f_1+1$.
\begin{enumerate} 
\item If $f_2=f_1$, then $\mathcal{U}_2$ is the permutation of $\mathcal{U}_1$ obtained by swapping the first two entries. Therefore $\mathcal{U}_2$ is a signature of $Q_{n-r}$.   

\item If $f_{2}=f_{1}+1$, then $\mathcal{U}_2$ is the permutation of $\mathcal{U}$ obtained by swapping the first two entries. Therefore $\mathcal{U}_2$ is a signature of $Q_{n-r}$. 

\item If $f_{2}> f_{1}+1$, then $f_{1}+1\leq f_{2}-1$. Let 
\[
\mathcal{U}_2=(f_{1}+1, f_{2}-1, f_{3}, \dots, f_{n-r})= (f_{1}'', f_{2}'', f_{3}'', \dots, f_{n-r}'').
\]
Then
\[
f_{1}''\leq f_{2}''< f_{3}''\leq \dots \leq f''_n
\]
is in nondecreasing order; and  
\[
\sum_{i=1}^k f''_i=
\begin{cases}
f_1+1\geq 3 &\text{for $k=1$};\\
\sum_{i=1}^k f_i\geq 2^k-1,                 &\text{for $2\leq k\leq n-r$},
\end{cases} 
\]
with equality in the second case when $k=n-r$.     
Therefore the signature condition is satisfied and we conclude that $\mathcal{U}_2$ is a signature of $Q_{n-r}$. 
\end{enumerate} 

Since $\mathcal{U}_1$ and $\mathcal{U}_2$ are signatures of $Q_{n-r}$, there are spanning trees $T_1$ and $T_2$ of $Q_{n-r}$ with signatures $\mathcal{U}_1$ and $\mathcal{U}_2$ respectively. Let $T'$ be the tree obtained from $T$ by replacing $T(R,X)$ with $T_1$,  and $T(R,Y)$ with $T_2$. Then $T'$ has signature $\mathcal{S}$, and the subtrees  $T'(R,X)$ and $T'(R,Y)$ have different signatures, as required.
\end{proof}

We now have everything we require to prove Theorem~\ref{thm:disconnected}.

\begin{proof}[Proof of Theorem~\ref{thm:disconnected}]
Without loss of generality, we may assume $\mathcal{S}$ is ordered with unsaturated part $\mathcal{S}'=(a_1, \dots, a_s)$. Let $1\leq r<s$ be the largest integer such that $\mathcal{S}'$ reduces over $R=[r]$, and choose distinct $X,Y\subseteq R$. By Remark~\ref{rem:choose.r} and Lemma~\ref{lem:differentsubsignatures}, there exists a spanning tree $T$ of $Q_n$ with signature $\mathcal{S}$ such that the subtrees $T(R,X)$ and $T(R,Y)$ have different signatures. Let $T'$ be the spanning tree obtained from $T$ by swapping $T(R,X)$ and $T(R,Y)$. Since the signatures of $T(R,X)$ and $T(R,Y)$ are invariant under edge slides by Corollary~\ref{cor:invariantsubsignatures}, the trees $T$ and $T'$ lie in different components of $\eslidesig{\mathcal{S}}$. It follows that $\eslidesig{\mathcal{S}}$ is disconnected, as claimed. 
\end{proof}

\section{Discussion}
\label{sec:discussion}

Theorem~\ref{thm:disconnected} shows that strict reducibility is an obstruction to being connected. We conjecture that this is the only obstruction to connectivity:

\begin{conjecture}
\label{conj:connectivity}
Let $\mathcal{S}=(a_1, \dots, a_n)$ be a signature of $Q_n$. Then the edge slide graph $\eslidesig{\mathcal{S}}$ is connected if and only if $\mathcal{S}$ is irreducible or quasi-irreducible.
\end{conjecture}

The ``only if'' direction of Conjecture~\ref{conj:connectivity} is Theorem~\ref{thm:disconnected}. As discussed below the ``if'' direction is known to be true for $n\leq 4$, for a certain class of irreducible signatures of $Q_5$, and for two infinite families of irreducible signatures. 
If true, the conjecture together with Observation~\ref{obs:contraction} and Theorem~\ref{thm:redR} would show that connected components of the edge slide graph of $Q_n$ are characterised in terms of signatures of spanning trees of subcubes of $Q_n$. 
We show below in Theorem~\ref{thm:connectivity-saturated} that it suffices to consider the case where $\mathcal{S}$ is irreducible only. 

The cases $n=1$ and $n=2$ are trivial. For $n\geq 3$ a useful approach is to reduce the problem to studying upright trees. By Tuffley~\cite[Cor.~15]{tuffley-2012} every tree is connected to an upright tree by a sequence of edge slides, so it suffices to show that every upright tree with signature $\mathcal{S}$ lies in a single component. Up to permutation there is a unique irreducible signature $(2,2,3)$ of $Q_3$, and using this approach it is straightforward to show that $\eslidesig{2,2,3}$ is connected. This is done by Henden~\cite{henden-2011}, who also determines the complete structure of $\eslidesig{2,2,3}$. 

For $n\geq 4$ the first author's doctoral thesis~\cite{alfran-2017}, completed under the supervision of the second and third authors, makes substantial partial progress towards an inductive proof of the conjecture. Al Fran~\cite[Defn~5.3.1]{alfran-2017} introduces the notion of a \emph{splitting signature} of $\mathcal{S}$ with respect to $n$. This is a signature $\mathcal{D}$ of $Q_{n-1}$ such that there exists an upright spanning tree $T$ of $Q_n$ such that $\sig(T)=\mathcal{S}$ and $\sig(T\cap Q_{n-1})=\mathcal{D}$.  As the culmination of a series of results Al Fran proves the following:

\begin{theorem}[Al Fran~{\cite[Thm~11.1]{alfran-2017}}]
\label{thm:reduction}
Let $n\geq 4$ and let $\mathcal{I}$ be an ordered irreducible signature of $Q_n$. Suppose that every irreducible signature of $Q_k$ is connected for all $k<n$. Suppose that $\mathcal{I}$ has an ordered irreducible splitting signature $\mathcal{D}$ with respect to $n$ such that every upright spanning tree with signature $\mathcal{I}$ and splitting signature $\mathcal{D}$ lies in a single component of $\eslide{n}$. Then the edge slide graph $\eslidesig{\mathcal{I}}$ is connected. 
\end{theorem}

This reduces the inductive step of a proof of Conjecture~\ref{conj:connectivity} to the problem of showing that every irreducible signature has a suitable splitting signature as given.  
Al Fran proves the existence of such a splitting signature for every  irreducible signature of $Q_4$, and (under the inductive hypothesis that every irreducible signature of $Q_{n-1}$ is connected) for every irreducible signature $\mathcal{I}=(a_1,\dots,a_n)$ admitting a \emph{unidirectional splitting signature}: a splitting signature $\mathcal{D}=(d_1,\dots,d_{n-1})$ such that $d_i=a_i$ for all but one index $i\leq n-1$. This proves the ``if'' direction of Conjecture~\ref{conj:connectivity} for $n=4$, and for irreducible signatures of $Q_5$ admitting a unidirectional splitting signature. Al Fran shows that when $\mathcal{I}$ does not admit a unidirectional splitting signature it admits a \emph{super rich splitting signature} (defined in terms of the excess), and conjectures such splitting signatures satisfy the requirements of Theorem~\ref{thm:reduction}.

Independently, Al Fran also proves the connectivity of two infinite families of irreducible signatures. For each $n\geq 3$ there is a unique ordered irreducible signature $\mathcal{I}_n^{(-1)}$ of $Q_n$ such that $\excess{k}{\mathcal{I}_n^{(-1)}}=1$ for all $k<n$; and for each $n\geq 4$ there is a unique ordered irreducible signature $\mathcal{I}_{(3,n)}^{(+1,-1)}$ with excess 2 for $k=2$, and excess 1 for $k<n$, $k\neq 2$. The first three members of these families are $(2,2,3)$, $(2,2,4,7)$, $(2,2,4,8,15)$; and $(2,3,3,7)$, $(2,3,3,8,15)$ and $(2,3,3,8,16,31)$, respectively. By~\cite[Thms~10.1.1 and~10.2.1]{alfran-2017} every signature in these families has a connected edge slide graph.

We conclude the paper with Theorem~\ref{thm:connectivity-saturated}, which reduces the quasi-irreducible case of Conjecture~\ref{conj:connectivity} to the irreducible case.

\begin{theorem}
\label{thm:connectivity-saturated}
Let the ordered signature $\mathcal{S}=(a_1, \dots, a_n)$ be saturated above direction $r$. Then $\mathcal{E}(\mathcal{S})$ is connected if and only if $\mathcal{E}(a_1, \dots, a_r)$ is connected. 

In particular, $\eslidesig{\mathcal{S}}$ is connected if and only if $\eslidesig{\unsat(S)}$ is connected. 
\end{theorem}

\begin{proof}
We may write $\mathcal{S}=(\mathcal{S}',2^{n-1})$, where $\mathcal{S}'=(a_1, \dots, a_{n-1})$. Inductively, it suffices to show that $\eslidesig{\mathcal{S}}$ is connected if and only if $\eslidesig{\mathcal{S}'}$ is connected.

A spanning tree $T$ of $Q_n$ with signature $\mathcal{S}$ contains every edge of $Q_n$ in direction $n$, and under the isomorphism
\[
\Psi_{[n-1]}:\eslidesig{\red{[n-1]}}\to\eslidesig{Q_n/\{n\}} \boxempty (\eslide{1})^{\boxempty 2^{n-1}} 
                     \cong\eslidesig{Q_n/\{n\}}
\]
it corresponds to the spanning tree $T/\{n\}$ of $Q_n/\{n\}$ with signature $\mathcal{S'}$ obtained by contracting these edges. The graph $Q_n/\{n\}$ has underlying simple graph $Q_{n-1}$, with two parallel edges labelled $\emptyset$ and $\{n\}$ for each edge of $Q_{n-1}$. Thus, $T/\{n\}$ in turn corresponds to the spanning tree $T'=\capr{[n-1]}(T)$ of $Q_{n-1}$ with signature $\mathcal{S}'$, together with a choice of label $\emptyset$ or $\{n\}$ on every edge. 
Moreover, by Theorem~\ref{thm:reducibleslides} an edge $e$ of $T$ or $T/\{n\}$ can be slid in direction $i\in [n-1]$ if and only if the corresponding edge $\capr{[n-1]}(e)$ of $T'$ can be, and the label $\emptyset$ or $\{n\}$ can be freely changed at any time. 

Suppose that $\eslidesig{\mathcal{S}'}$ is connected, and let $T_1,T_2\in\eslidesig{\mathcal{S}}$. Since $\eslidesig{\mathcal{S}'}$ is connected there is a sequence of edge slides transforming $\capr{[n-1]}(T_1)$ into $\capr{[n-1]}(T_2)$. These edge slides may all be carried out in $Q_n$, to give a sequence of edge slides from $T_1$ to a tree $T_2'$ such that $\capr{[n-1]}(T_2')=\capr{[n-1]}(T_2)$. The trees $T_2, T_2'$ may differ only in the edge labels $\emptyset$ or $\{n\}$, and after a further series of edge slides in direction $n$ only these can be brought into agreement. Therefore $\eslidesig{\mathcal{S}}$ is connected. 

Conversely, suppose $\eslidesig{\mathcal{S}}$ is connected, and let $T_1,T_2\in\eslidesig{\mathcal{S}'}$. Choose spanning trees $T_1',T_2'$ of $Q_n$ such that $\capr{[n-1]}(T_i')=T_i$ for each $i$ (for example, by regarding $Q_{n-1}$ as a subgraph of $Q_n$, and adding all edges of $Q_n$ in direction $n$ to $T_i$ for each $i$). There is a sequence of edge slides in $Q_n$ transforming $T_1'$ into $T_2'$, and applying $\capr{[n-1]}$, these may all be carried out in $Q_{n-1}$ to transform $T_1$ into $T_2$.  Therefore $\eslidesig{\mathcal{S}'}$ is connected also. 
\end{proof}

\end{document}